\documentclass[leqno,11pt]{amsart}
\usepackage{amsmath}
\usepackage{amsfonts}
\usepackage{amssymb}
\usepackage{amsthm}

\usepackage{mathrsfs}
\usepackage{dsfont}
\usepackage{txfonts}
\usepackage{mathabx}

\usepackage{graphicx}
\usepackage[all]{xy}
\usepackage{pgf,tikz}
\usepackage{subfigure}
\usepackage{tikz-3dplot}
\usetikzlibrary{patterns,arrows}

\usepackage[pdftex]{hyperref}
\usepackage{color}
\usepackage[shortlabels]{enumitem}

\newtheorem{theorem}{Theorem}[section]
\newtheorem{lemma}[theorem]{Lemma}

\theoremstyle{definition}
\newtheorem{definition}[theorem]{Definition}
\newtheorem{notation}[theorem]{Notation}

\theoremstyle{remark}
\newtheorem{remark}[theorem]{Remark}

\newcommand{\bd}{\partial}

\newcommand{\rno}{\mathbb{R}^{n+1}}
\newcommand{\rtw}{\mathbb{R}^2}
\newcommand{\rth}{\mathbb{R}^3}

\newcommand{\cv}{\text{Conv}}

\newcommand{\cL}{\mathcal{L}}
\newcommand{\cA}{\mathcal{A}}
\newcommand{\cR}{\mathcal{R}}

\title{Translating solutions to the Gauss curvature flow with flat sides}

\author{Kyeongsu Choi}
\address{ {\bf Kyeongsu Choi:} Department of Mathematics, Massachusetts Institute of Technology, 77 Massachusetts Ave,  \newline \hphantom \quad\, Cambridge, MA 02139, USA.}
\email{choiks@mit.edu}
\author{Panagiota Daskalopoulos}
\address{ {\bf P. Daskalopoulos:} Department of Mathematics, Columbia University, 2990 Broadway, New York, NY 10027, USA.}
\email{pdaskalo@math.columbia.edu}
\author{Kiahm Lee}
\address{ {\bf Ki-Ahm Lee:} Department of Mathematical Sciences, Seoul National University, Seoul 151-747, Korea, 
\&  Korea\newline \hphantom \quad\,   Institute for Advanced Study, Seoul 130-722, Korea.}
\email{kiahm@snu.ac.kr}

\begin{document}

\maketitle

\begin{abstract}
We derive local $C^{2}$ estimates for complete non-compact  translating solitons of  the Gauss curvature flow in $\rth$ which are graphs 
over a convex domain  $\Omega$.  This is closely  is related to deriving local $C^{1,1}$ estimates for the degenerate Monge-Amp\'ere equation. 
As a result, given a weakly convex bounded domain $\Omega$,  we establish the existence of a $C^{1,1}_{\text{loc}}$ translating soliton. In particular, when the boundary $\partial \Omega$ has a line segment, we show the existence of flat sides of the translator from  a local   a'priori  non-degeneracy estimate near the free-boundary.   

\end{abstract}

\section{introduction}

We recall that  an one-parameter family of immersions $F:M^n \times (0,T) \to \rno$ is a solution of the Gauss curvature flow, if for 
each $t \in (0,T)$, $F(M^n,t)=\Sigma_t$ is a complete convex hypersurface embedded in $\rno$  satisfying
\begin{equation*}\label{eq:INT GCF}
\frac{\bd}{\bd t}  F(p,t)= K(p,t) \,  \vec{n}(p,t)  \tag{1.1}
\end{equation*} 
where $K(p,t)$ and $\vec{n}(p,t)$ are the Gauss curvature and the interior unit vector of $\Sigma_t$ at the point $F(p,t)$, respectively. 

In this paper, we consider a translating solution $\Sigma_t$ to the Gauss curvature flow in $\rth$ satisfying
\begin{align*}
\Sigma_t=\Sigma+ct\, \vec{e}_3 \eqqcolon \{Y + c t\, \vec{e}_3 \in \rth  : Y \in \Sigma \},
\end{align*}
where $ \vec{e}_3 =(0,0,1)$ and	 the {\em speed} $c$ is a constant, and $\Sigma$ is a complete convex hypersurface embedded in $\rth$. We observe that there exist a convex open set $\Omega \subset \rtw$ and a convex function $u:\Omega \to \mathbb{R}$ satisfying
\begin{center}
\begin{large}
\textit{ $\Sigma$ is the boundary of\, $\{(x,t):x\in \Omega, t \geq u(x)\}$.}
\end{large}
\end{center}
By the result in \cite{Urbas98}, the set $\Omega$ must be {\em  bounded }. 
If  $\cA(\Omega)$ denotes the area of $\Omega$, it follows that   $u$ is a smooth function satisfying
\begin{equation}\label{eq:INT GCFeq}
\begin{cases}
&\displaystyle \frac{\det  D^2 u}{(1+|Du|^2)^{\frac{3}{2}} } =\frac{2\pi}{\cA(\Omega)} \quad \text{in} \; \Omega, \\
& \displaystyle \lim_{x \to \bd\Omega} |Du|(x)  =  +\infty \quad \text{on}  \;\bd\Omega.
\end{cases}\tag{1.2}
\end{equation}

Conversely, given an open bounded convex set $\Omega \subset \rtw$, there exists a solution $u: \Omega \to \mathbb{R}$ of \eqref{eq:INT GCFeq}, and any two solutions differ by a constant. (See \cite{Urbas98} and Theorem 4.8 in \cite{Urbas88}).  Hence, given a translator $\Sigma$ in $\rth$ of the Gauss curvature flow, there exists an open  bounded  convex set $\Omega \in \rtw$ such that $\Sigma$ converges to the cylinder $\bd\Omega \times \mathbb{R}$, and the immersion $F:M^2 \to \rth$ of $F(M^2)=\Sigma$ satisfies
\begin{equation*}\label{eq:INT GCF}
K(p)= \frac{2\pi}{\cA(\Omega)}\,\langle \vec{n}(p),\vec{e}_3 \;\rangle. \tag{*}
\end{equation*}

We recall the result of John Urbas in  \cite{Urbas98}.

\begin{theorem}[Urbas]\label{thm:INT Urbas}
Given an open bounded convex  domain $\Omega \subset \rtw$, there exists a convex solution $u :\Omega \to \mathbb{R}$  satisfying \eqref{eq:INT GCFeq}, and it is unique up to addition by a constant. In particular, if for each $x_0\in \bd \Omega$, there exists a ball $B \subset \rtw$ satisfying $\Omega \subset B$ and $x_0 \in \bd B$, then the solution $u$ is a smooth function satisfying
\begin{equation*}
\lim_{x\to \bd \Omega}u(x)=+\infty.
\end{equation*}
\end{theorem}

\bigskip

This  result guarantees that there exists a unique $C^1$ translator $\Sigma=\bd\{(x,t):x\in \Omega, t \geq u(x)\}$ for any open bounded convex domain $\Omega$. Also, if  $\Omega$ is a {\em uniformly convex} domain, then $\Sigma$ is  {\em strictly  convex}, and thus 
$C^\infty$ smooth by standard estimates.  However, if $\Omega$ is weakly convex, then $\Sigma$ may not be  strictly  convex  on the boundary of 
$\Omega$. Richard Hamilton conjectured that if $\Omega$ is a {\em square}, then $\Sigma$ has {\em flat sides}  on the boundary of $\Omega$. 
This is shown in the next  picture.

\def\Aangle{81}
\def\Bangle{0}
\def\Cangle{0}
\def\AAangle{0}
\def\BAangle{20}
\def\CAangle{-5}

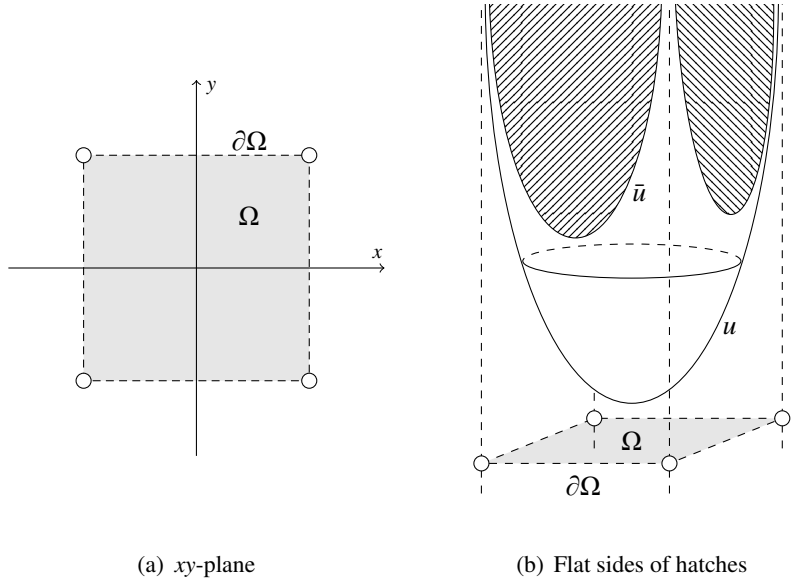
\begin{figure}[h]
\subfigure[$xy$-plane]{
\begin{tikzpicture}\label{fig:INT Domain}
[x=1cm,y=1cm] \clip(-1,-2) rectangle (4,5);
\draw[fill=gray!20!,densely dashed] (0,0)--(0,3)--(3,3)--(3,0)--(0,0);
\draw[fill=white]  (0,0) circle (0.1cm);
\draw[fill=white] (0,3) circle (0.1cm);
\draw[fill=white]  (3,0) circle (0.1cm);
\draw[fill=white]  (3,3) circle (0.1cm);
\draw[->] (1.5,-1) -- (1.5,4);
\draw[->] (-1,1.5) -- (4,1.5);
\begin{scriptsize}
\draw[color=black] (2.2,2.2) node[scale=1.3] {$\Omega$};
\draw[color=black] (2.2,3.2) node[scale=1.3] {$\bd\Omega $};
%\draw[color=black] (-0.2,3.2) node[scale=1.3] {$V$};
\draw[color=black] (1.7,3.9) node {$y$};
\draw[color=black] (3.9,1.7) node {$x$};
\end{scriptsize}
\end{tikzpicture}
}
\subfigure[Flat sides of hatches]{
        \begin{tikzpicture}\label{fig:INT Flat sides}
[x=1cm,y=1cm] \clip(-3,-5) rectangle (3,2);
\tikzset{xyplane/.estyle={cm={
cos(\Bangle),sin(\Bangle)*sin(\Aangle),sin(\Cangle)*sin(\Bangle),
cos(\Cangle)*cos(\Aangle)-sin(\Cangle)*cos(\Bangle)*sin(\Aangle),(0,0)}}}
\tikzset{Fplane/.estyle={cm={
cos(\BAangle),sin(\BAangle)*sin(\AAangle),sin(\CAangle)*sin(\BAangle),
cos(\CAangle)*cos(\AAangle)-sin(\CAangle)*cos(\BAangle)*sin(\AAangle),(1.25,2.4)}}}
\draw[dashed,fill=gray!20!] (2,-3.5)--(-0.5,-3.5)--(-2,-4.1)--(0.5,-4.1)--(2,-3.5);
\draw[dashed] (-0.5,-3.9) -- (-0.5,3);
\draw[dashed] (-0.5,-3) -- (-1.5,0);
\draw[fill=white] (0,2.2) ellipse (1.95cm and 5.5cm);
\draw[dashed] (-2,-4.5) -- (-2,3);
\draw[dashed] (0.5,-4.5) -- (0.5,3);
\draw[dashed] (2,-3.9) -- (2,3);
\draw[pattern=north east lines] (-0.75,2.4) ellipse (1.15cm and 3.5cm);
\draw[xyplane,dashed] (1.45,-9) arc (0:180:1.45cm);
\draw[xyplane] (-1.45,-9) arc (180:360:1.45cm);
\draw[Fplane,pattern=north west lines] (-0.03,0) ellipse (0.7cm and 3.2cm);
\draw[fill=white]  (2,-3.5) circle (0.1cm);
\draw[fill=white]  (-0.5,-3.5) circle (0.1cm);
\draw[fill=white]  (-2,-4.1) circle (0.1cm);
\draw[fill=white]  (0.5,-4.1) circle (0.1cm);
\begin{scriptsize}
\draw[color=black] (0,-3.8) node[scale=1.3] {$\Omega$};
%\draw[color=black] (-2.3,-4) node[scale=1.3] {$V$};
\draw[color=black] (-0.65,-4.4) node[scale=1.3] {$\bd\Omega $};
\draw[color=black] (1.3,-2.3) node[scale=1.3] {$u$};
\draw[color=black] (0.1,-0.5) node[scale=1.3] {$\bar u$};
\end{scriptsize}
\end{tikzpicture}
}
\caption{Translator $\Sigma$ on a square}\label{fig:INT Solution}
\end{figure}

The above picture describes that in the case that $\Omega$ is the square and $L$ is one of the  edges of  $\partial \Omega$, for any $x_0 \in L$
the limit $\bar u(x_0) := \lim_{x \to x_0} u(x)$ exists and defines a function $\bar u: L \to \mathbb{R}$ whose graph is the boundary of one of the flat sides.

The main result in this paper concerns with the general case of translators in $\mathbb{R}^3$ with flat sides and in particular provides  the proof of Hamilton's conjecture and establishes the {\em optimal regularity}  of the translator in this degenerate case. The result is given for surface solutions in $\mathbb{R}^3$ but it can be generalized  in higher dimensions under certain constraints. See the discussion at the end of this section for further remarks about the higher dimensional case.

\begin{theorem}\label{thm:INT Flat sides}
Let $\Omega$ be a convex open bounded domain in $\mathbb{R}^2$, and let $u$ be a solution to \eqref{eq:INT GCFeq} on $\Omega$. Then, the corresponding solution $\Sigma$ to \eqref{eq:INT GCF} is  of class $C^{1,1}_{\text{loc}}$.

 Suppose that the boundary $\bd \Omega$ contains a line segment $L=\{tp+(1-t)q:p,q \in \bd \Omega , t\in (0,1)\}$. Then, there exists a function $\bar u:L \to \mathbb{R}$ such that
\begin{equation*}
\bar u(x_0)= \lim_{x \to x_0}u(x).
\end{equation*}
\end{theorem}

\medskip

\bigskip The Gauss curvature flow was first introduced by Firey in \cite{Firey},   where he  showed that a closed  strictly convex and centrally symmetric solution in $\rth$  converges to a round point. In \cite{Tso} Tso established the existence of  closed and strictly convex solutions in $\rno$ and showed that it converges to a point. Andrews \cite{Andrews00}  extended Tso's  result  to the flow by positive powers of the Gauss curvature, namely a strictly convex closed solution, to the $\alpha$-Gauss curvature flow $\bd_t F=K^{\alpha}\vec{n}$.

The asymptotic behavior of strictly convex closed solutions has been widely studied. In the affine-invariant case $\alpha = \frac{1}{n+2}$, Calabi \cite{Calabi} showed that closed self-similar solutions are ellipsoids, and Andrews \cite{Andrews96} established the convergence of rescaled solutions to the ellipsoids. In the higher power case $\alpha>\frac{1}{n+2}$, the convergence of rescaled solutions to the round spheres was established by Chow \cite{Chow} for $\alpha=\frac{1}{n}$, Andrews \cite{Andrews99} for $\alpha=1$ and $n=2$, and Andrews-Chen \cite{Andrews-Chen} for $\alpha \in (\frac{1}{2},1)$ and $n=2$. 

Recently, Guan-Ni \cite{Guan-Ni} showed the convergence of the rescaled flows to closed self-similar solutions for $\alpha=1$ and $n\geq 2$, and jointly with Andrews \cite{Andrews-Guan-Ni} they extended the result for $\alpha>\frac{1}{n+2}$. The uniqueness of strictly convex closed solutions was proven by Choi-Daskalopoulos \cite{Choi-Daskalopoulos2} for $\frac{1}{n}<\alpha < 1+\frac{1}{n}$, and the result was extended for $\alpha >\frac{1}{n+2}$ in their joint work with Brendle \cite{Brendle-Choi-Daskalopoulos}. In the work \cite{Brendle-Choi-Daskalopoulos}, they also gave an alternate proof of the result in \cite{Calabi} for $\alpha=\frac{1}{n+2}$.

\bigskip

Regarding the non-comapct case, the all-time existence of a non-compact complete and strictly convex solution $\alpha$-Gauss curvature flow  in $\rno$ was established by Choi-Daskalopoulos-Kim-Lee \cite{Choi-Daskalopoulos-Kim-Lee}. Therefore, the convergence of such solutions to complete self-similar solutions becomes a natural question. Different from the closed case, self-expanders for $\alpha>0$ and translators $\alpha >\frac{1}{2}$ have been  completely classified  by Urbas \cite{Urbas98} in $C^1$ sense. In this paper, we will improve the $C^1$ regularity to $C^{1,1}$ for $n=2$ and $\alpha=1$, and we will show that the $C^{1,1}$ is the optimal regularity. See Remark \ref{rmk:INT Optimal Regularity}. As our main result states, we are especially interested in translators with flat sides.

Closed solutions of the Gauss curvature flow  in $\rth$ with  flat sides was considered by R. Hamilton in \cite{Hamilton}, and the
$C^\infty$  regularity of its free boundary was studied in \cite{Daskalopoulos-Hamilton, Daskalopoulos-Lee, Kim-Lee-Rhee}. The optimal $C^{1,1}$ regularity for $n=2$ and $\alpha=1$ was obtained in \cite{Andrews99}, and the $C^{1,\beta}$ regularity for other $n$ and $\alpha$ was established in \cite{Daskalopoulos-Savin}.

\smallskip

\noindent{\em Discussion on the Proof: }  

\smallskip
To establish the interior $C^{1,1}$ regularity of the surface $\Sigma$ we will bound $\eta \, \Lambda$,  where 
$\Lambda= \max (\lambda_1(p), \lambda_2(p))$  denotes the largest principal curvature of $\Sigma$ 
and $\eta$ is the cut-off function $\eta=(|F(p)|^2-R^2)_+$ as defined in Notation \ref{not:INT Op}. If one chooses 
instead the  standard cut-off function $\eta_s(p) = (R^2-|F(p)|^2)_+$ or the  level set cut-off function $\eta_L(p)=(M-\langle \vec{e}_3,\vec{n}(p)\rangle)_+$, then the estimate fails because of counter examples. 
Since  the equation \eqref{eq:INT GCF} only depends on the area $\cA$, the solution $\Sigma_1$ defined on $(0,1)\times (0,1)$ satisfies the same equation to the solution $\Sigma_\epsilon$ defined on $(0,\epsilon)\times (0,1/\epsilon)$. However, it is clear that $\Sigma_\epsilon$ has larger curvature than $\Sigma_1$ near the tip. However, $\eta_s$ and $\eta_L$ can not distinguish 
between $\Sigma_1$ and $\Sigma_\epsilon$. Thus, one  should use a cut-off function which included information about  the global structure $\Omega$ as the cut-off function $\eta$ in Notation \ref{not:INT Op}.

\smallskip
To show the existence of flat sides, we need to  deal with technical  difficulties arising from the non-compactness of our solution $\Sigma$. 
One of the difficulties is the {\em double degeneracy}  on the flat sides. We may consider $(0,1,0) \in \rth$ as the height vector, and define a convex function $h(x,z)$ whose graph $(x,h(x,z),z)$ is the lower part of the solution $\Sigma$. Then, $h(x,z)$ satisfies
\begin{align*}\label{eq:INT Horizontal Eq}
\frac{\det D^2 h}{(1+|Dh|^2)^{\frac{3}{2}}}=-\frac{2\pi}{\cA(\Omega)}D_z h.\tag{1.3}
\end{align*}
Then, near the flat side $\{(x,z):h=-1\}$, the right hand side $-D_z h=|Dh| \, \langle -v,e_2\rangle$ has two degenerate factors $|Dh|$ and $\langle -v,e_2\rangle$, where $e_2=(0,1)$ and $v \in \rtw$ is the outward normal direction of the level set of $h$. Since the level sets of $h$ becomes parallel to $e_2$ at the infinity, $\langle -v,e_2\rangle$ goes to zero at the infinity.

Another challenge comes from the fact  that the non-degeneracy of $|Dh|$ depends on the global structure $\Omega$. In the previous works \cite{Daskalopoulos-Hamilton, Daskalopoulos-Lee, Hamilton, Kim-Lee-Rhee}, the lower bound for $|Dh|\, h^{-\frac{1}{2}}$ depends on the initial data, and this does not apply  in the elliptic setting. Hence, we have to develop a new  non-degeneracy estimate for a solution $h$ to the horizontal equation \eqref{eq:INT Horizontal Eq}  which includes the information of the global structure of the domain $\Omega$. We will consider the gradient bound of a solution $u$ to \eqref{eq:INT GCFeq} instead of the non-degeneracy of a solution $h$ to \eqref{eq:INT Horizontal Eq} for convenience.

\medskip

\noindent{\em Outline of the paper:}  

\smallskip

A brief {\em outline} of this paper is as follows : In section 2 we will summarize the notation which will be used throughout 
the paper.  In section 3, we will derive  interior $C^2$ estimates  for any strictly convex complete smooth solution $\Sigma$
by using a  Pogorelov type computation which we  localize by introducing the cut-off function $\eta$ mentioned above. 
 If a portion of the boundary $\bd\Omega$ is sufficiently close  to a line segment $L$ with  normal direction $e$, then  by using our $C^2$ estimate, we will show that the solution $\Sigma$ can be written as a  graph with respect to the direction $e$ in a certain region including $L\times \mathbb{R}$. 

The proof of the $C^2$ estimate although somehow technical, follows known  techniques. The main challenge 
 and new feature  of  this work lies in the proof of 
the existence of the flat side  in Sections 4 and 5.  
Section 4 contains one of the main points in this paper which follow from the  construction of  a barrier. Since a complete supersolution defined on $\Omega'$ satisfies $\cA(\Omega') > \cA(\Omega)$, we can not put such a supersolution on top of a our solution $\Sigma$. Instead, we need to cut a complete supersolution 
and slide it to contact $\Sigma$. This way we  obtain the gradient bound of the level set at the contact point, which leads to the local lower bound $\langle -v,e_2\rangle$. This implies the crucial partial derivative estimate of a solution $u$ to \eqref{eq:INT GCFeq} in Theorem \ref{thm:PGE Gradient bound for level set}.  
In section 4, we derive a separable equation from the equation \eqref{eq:INT GCFeq}, namely 
\begin{align*}
\frac{u_{xx}u_{yy}}{(1+u_y^2)^{\frac{3}{2}}} \geq \frac{\det D^2u}{(1+|Du|^2)^{\frac{3}{2}}} =\frac{2\pi}{\cA(\Omega)}.
\end{align*}
By integrating this equation, we obtain the gradient bound at a certain point $(x_0,-1+\epsilon)$ in Section 5, Lemma \ref{lemma:DTF Gradient estimate}. Then, we establish the distance between the tip and the flat side bound  in Theorem \ref{thm:DTF Distance}. The existence of the flat side then follows and is shown in Theorem 
\ref{thm:DTF main thm}.

\begin{remark}[Optimal regularity]\label{rmk:INT Optimal Regularity}
Let us briefly remark  that if there exists a flat side on a translator $\Sigma$, then $\Sigma$ has at most $C^{1,1}$ regularity. The horizontal equation \eqref{eq:INT Horizontal Eq} yields
\begin{align*}
h_{vv}h_{\tau\tau} \geq \frac{\det D^2 h}{(1+|Dh|^2)^{\frac{3}{2}}} \geq -\frac{2\pi}{\cA}|Dh|\langle v,e_2 \rangle
\end{align*}
where $v(x,z)$ is the outward normal and $\tau(x,z)$ is a tangential direction of the level set of $h(x,z)$ at a point $(x,z)$. We denote by $L_{r}$ the level set $\{(x,z):h(x,z)=r \}$, and denote by $\kappa(x,z)$ the curvature of $L_{h(x,z)}$ at $(x,z)$. Since we have $h_{\tau\tau}=|Dh| \, \kappa$, the inequality above gives $h_{vv} \, \kappa \geq  -2\pi \cA^{-1}\langle v,e_2 \rangle$. We will establish the local lower bound for $-\langle v,e_2 \rangle$ in section 3, which guarantees that 
\begin{align*}
h_{vv}\kappa \geq c.
\end{align*}
We choose a neighborhood $U$ of a point $(x_0,z_0)$ on the free boundary $\Gamma$, namely $(x_0,z_0) \in \Gamma \eqqcolon \bd L_{-1}$. Since the level sets $L_r$ monotonically converge to $\Gamma$, there exists a constant $c$ such that $\int_{L_r \cap U} ds \geq c$ for $r$ close enough to $-1$, where $s$ is the arc length parameter. Hence, the following holds
\begin{align*}
2 \pi \geq \int_{L_r} \kappa ds \geq \int_{L_r \cap U} \kappa ds \geq \frac{c}{\max_{L_r \cap U} h_{vv}}.
\end{align*}
Thus, $\max_{L_r \cap U} h_{vv} \geq c$ holds for some uniform constant $c$. However, we have $D^2h=0$ on $L_{-1}$. Therefore, $D^2h$ is not a continuous function.
\end{remark}

\smallskip

\noindent{\em Discussion  on the  higher dimensional case: }  

\smallskip

We consider a bounded open convex domain $\Omega \in\mathbb{R}^n$ with a flat side $L\subset \bd\Omega$, namely $L$ is an open convex set of a hyperplane. The results in sections 4 and 5 can be naturally extended to higher dimensions. So, we need to show that approximated strictly convex solutions are graphs with respect to the normal direction $e$ of $L$ in a certain uniform region including $L\times \mathbb{R}$.  Therefore, if $\Omega$ has the axial symmetry with respect to the direction $e$, then the existence of a flat side on $\Sigma$ readily follows.

For the case without the symmetry, we need uniform estimates for the size of the region where the approximated solutions are graphs with respect to $e$. In section 3, we utilize a local $C^{1,1}$ estimate to derive the size estimate. However, the translator would have at most local $C^{1,\frac{1}{n-1}}$ regularity as the parabolic case. Therefore, it would be an interesting question to find an appropriate partial regularity of the translator in higher dimensions, which yields the desired size estimate for the region  where the solution is a graph with respect to $e$.

\section{Notation}
For the convenience of the reader, we give below some basic notation which will be frequently used in what follows. We will use Definition \ref{def:INT cutting ball} and Notation \ref{not:INT Op} in section 3. Definition \ref{def:INT Axes of sym} and Notation \ref{not:INT Non-Dgn} will be used in section 4 and 5.

\begin{notation}[To be used in section 3]\label{not:INT Op} We will use some standard notation on the metric, second fundamental form and the linearized operator. 

\begin{enumerate} 
\item Let $F:M^2\to \rth$ be an immersion defining a smooth and complete surface $\Sigma$ by $F(M^2)=\Sigma$. If $B_R(Y)$ cut $\Sigma$ as Definition \ref{def:INT cutting ball}, we define a cut-off function $\eta:M^2 \to \mathbb{R}$ by
 $$\eta(p)=(|F(p)-Y|^2-R^2)_+.$$

\item We recall that $g_{ij} = \langle F_i, F_j \rangle$, where $F_i \coloneqq \nabla_i F$. Also, we denote as usual by $g^{ij}$  the inverse matrix of
$g_{ij}$ and $F^i=g^{ij} \, F_j$. 

\item For a strictly convex smooth hypersurface $\Sigma_t$, we denote by $b^{ij}$ the the inverse matrix $(h^{-1})^{ij}$ of its  {\em second fundamental form}  $h_{ij}$, namely $b^{ij}h_{jk}=\delta^i_k$.

\item We denote by $\cL$  the {\em linearized }   operator 
$$\cL =  K b^{ij}\nabla_i \nabla_j \, .$$
Furthermore, $\langle \;, \;\rangle_\cL$ denotes the associated inner product 
$\displaystyle \langle \nabla f,\nabla g \rangle_\cL =   K b^{ij}\nabla_i f \nabla_j g$, 
 where $f,g$ are differentiable functions on $M^n$, and $\|\cdot\|_\cL $ denotes the $\cL$-norm given by the inner product $\langle \;, \;\rangle_\cL$ 

\item $H$ and $\Lambda$ denote the {\em  mean curvature} and the {\em largest principal curvature}, respectively.
\end{enumerate}

\end{notation}

\begin{definition}[Cutting ball]\label{def:INT cutting ball}
Given a ball $B_R(Y) \subset \rth$ and a complete surface $\Sigma\subset \rth$ , we say that a compact surface $\Sigma_c$ with boundary $\bd\Sigma_c$ is cut off from $\Sigma$ by $B_R(Y)$, if $ \Sigma_c \subset \Sigma$ and $\bd\Sigma_c \subset \bd B_R(Y)$ hold.
\end{definition}

\def\Aangle{87}
\def\Bangle{0}
\def\Cangle{0}
\begin{figure}[h]
\begin{tikzpicture}\label{fig:INT Cutting ball}
[x=1cm,y=1cm] \clip(-12.5,-1.2) rectangle (-0.5,1.2);
\tikzset{Tplane/.estyle={cm={
cos(\Bangle),sin(\Bangle)*sin(\Aangle),sin(\Cangle)*sin(\Bangle),
cos(\Cangle)*cos(\Aangle)-sin(\Cangle)*cos(\Bangle)*sin(\Aangle),(0,0)}}}
\draw[Tplane,dashed] (0,12) arc (90:120:12cm);
\draw[Tplane,dashed] (0,-12) arc (270:240:12cm);
\draw[Tplane] (12*cos 120,12*sin 120) arc (120:240:12cm);
\draw (-5.17,0) circle (1cm);
\draw[fill=black] (-5.17,0) circle (0.05cm);
\draw (-5.17,0) -- (-5.17+cos 30,sin 30) ;
\begin{scriptsize}
\draw[color=black] (-5.4,0) node[scale=1.3] {$Y$};
\draw[color=black] (-4.6,0.1) node[scale=1.3] {$R$};
\draw[color=black] (-6.7,0) node[scale=1.3] {$B_R(Y)$};
\draw[color=black] (-3.8,0.4) node[scale=1.3] {$\Sigma$};
\draw[color=black] (-9,0.6) node[scale=1.3] {$\Sigma_c$};
\end{scriptsize}
\end{tikzpicture}
\caption{Cutting ball}
\end{figure}
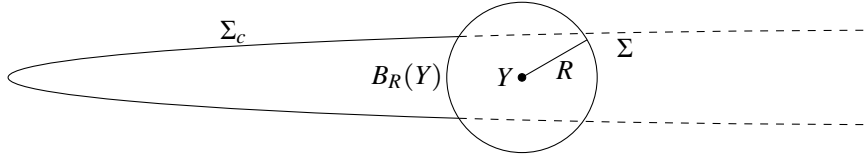

\begin{definition}[Axial symmetry]\label{def:INT Axes of sym}
We say that a surface $\Sigma \subset \rth$ has {\em axial  symmetry}, if $(x,y,z) \in \Sigma$ guarantees $(-x,y,z), (x,-y,z) \in \Sigma$. Similarly, a set $\Omega \subset \rtw$ has {\em axial symmetry}, if $(x,y)\in \Omega$ guarantees $(-x,y),(x,-y)\in \Omega$.
\end{definition}

\begin{notation}[To be used in sections 4 and 5]\label{not:INT Non-Dgn} Also, we  summarize some further notation. 
\begin{enumerate}
\item Given a set $A \subset \rth$ and a constant $s$, we denote the $x=s$ level set by $L^x_s(A)=\{(s,y,z)\in A\} $. Similarly, we denote $y=s$ and $z=s$ level set by $L^y_s(A)$ and $L^z_s(A)$, respectively.

\item Given a constant $s$ and a function $f:\Omega \to \mathbb{R}$ with $\Omega \subset \rtw$, we denote by $L_s(f)$ the $s$-level set $\{(x,y)\in \Omega : f(x,y)=s\}$.

\item We let $e_1$ and $e_2$ the unit vectors $(1,0)$ and $(0,1)$, respectively.

\item For a complete and convex curve $\Gamma \subset \rtw$, its convex hull $\cv(\Gamma)$ is given by $$\cv(\Gamma) = \{(t x+(1-t)y:x,y \in \Gamma, t \in [0,1]\}.$$ If $A$ is a subset of $\cv(\Gamma)$, then we say $A$ is enclosed by $\Gamma$ and use the notation 
$$A \prec \Gamma .$$ 

\item Given a set $A \subset \rtw$, $\text{cl} (A)$ and $\text{Int} (A)$ mean the closure and the interior of $A$, respectively. 
\end{enumerate}
\end{notation}

\section{Optimal $C^{1,1}$ regularity}

In this section, we will establish a local curvature estimate for smooth strictly convex complete solutions 
of equation  \eqref{eq:INT GCF}. In the last section we will use this estimate to obtain the optimal $C^{1,1}$ regularity for a weakly convex solution
of  \eqref{eq:INT GCF} in the degenerate case.   
We recall that a solution of \eqref{eq:INT GCF} has an immersion $F:M^2\to \rth$ of $F(M^2)=\Sigma$. Given a ball $B_R(Y)$  we define the  associated cut-off function $\eta$  by $\eta(p)=(|F(p)-Y|^2-R^2)_+$.  We have the following result. 

\begin{theorem}[Curvature bound]\label{thm:OP Optimal regularity}
Let $\Sigma$ be a smooth strictly convex complete solution of \eqref{eq:INT GCF}. Let  $\Sigma_c$ be the  cut off from $\Sigma$ by a ball $B_R(Y) \subset \rth$ as defined in Definition \ref{def:INT cutting ball}. Then, for any $p\in M^2$ with $F(p) \in \Sigma_c$, the maximum principal curvature 
$\Lambda(p) := \max \{ \lambda_1(p), \lambda_2(p) \}$ satisfies 
\begin{align*}
\eta \Lambda (p) \leq \frac{9\pi}{\cA(\Omega)}\sup_{F(q) \in \Sigma_c} |F(q)-Y|^{3}.
\end{align*}
\end{theorem}

\begin{proof}
We may assume, without loss of generality, that  $Y=0$. Recall the definition of the cutting ball as shown in Figure 2 above. 
The continuous function $\eta\Lambda$ attains its maximum on the compact set $\Sigma_c$ at some point $F(p_0) \in \Sigma_c$,
\begin{align*}
\eta\Lambda(p_0)= \max_{F(p)\in\Sigma_c}\eta\Lambda(p).
\end{align*}
Then, because we have $\eta=0$ on $\bd\Sigma_c$, $F(p_0)$ is an interior point of $\Sigma_c$. Thus, $\eta\Lambda$ attains a local maximum at $p_0$. Moreover, we can choose an open chart $(U,\varphi)$ with $p_0 \in \varphi(U)$ and $F(\varphi (U))\subset \Sigma_c$ such that the covariant derivatives $\{\nabla_1 F(p_0),\nabla_2 F(p_0)\}$ form an orthonormal basis of $T\Sigma_{F(p_0)}$ satisfying
\begin{align*}
&& g_{ij}(p_0)=\delta_{ij}, && h_{ij}(p_0)=\delta_{ij}\lambda_i(p_0), && \lambda_1(p_0)=\Lambda(p_0).
\end{align*}  
Next, we define the function $w:U \to \mathbb{R}$ by
\begin{align*}
w=\eta \frac{h_{11}}{g_{11}}.
\end{align*}
Then, the Euler formula guarantees $w \leq \eta \Lambda$ (c.f. Proposition 4.1 in \cite{Choi-Daskalopoulos1}). Therefore, for all $p \in U$, the following holds
\begin{align*}
w(p) \leq \eta \Lambda(p) \leq \eta \Lambda(p_0)=w(p_0).
\end{align*}
Thus, $w$ also attains its maximum at $p_0$.

Now, we consider the derivative of $w$. Then, $\nabla g_{11}=0$ gives
\begin{align*}\label{eq:OP Gradient}
\frac{\nabla_i w}{w}=\frac{\nabla_i h_{11}}{h_{11}}+\frac{\nabla_i \eta}{\eta} \;. \tag{3.1}
\end{align*}
Differentiating the equation above yields
\begin{align*}
\frac{\nabla_i\nabla_j w}{w}-\frac{\nabla_i w\nabla_j w}{w^2}=
\frac{\nabla_i\nabla_j h_{11}}{h_{11}}-\frac{\nabla_i h_{11}\nabla_j h_{11}}{(h_{11})^2}+
\frac{\nabla_i\nabla_j \eta}{\eta}-\frac{\nabla_i \eta\nabla_j \eta}{\eta^2} \;
\end{align*}
and  multiplying  by $Kb^{ij}$, we obtain
\begin{align*}\label{eq:OP 2nd order derivative}
\frac{\cL\, w}{w}-\frac{\|\nabla w\|_{\cL}^2}{w^2}=\frac{\cL \,h_{11}}{h_{11}}-\frac{\|\nabla h_{11}\|_{\cL}^2}{(h_{11})^2}+\frac{\cL \,\eta}{\eta}-\frac{\|\nabla \eta\|_{\cL}^2}{\eta^2}\;.\tag{3.2}
\end{align*}
Observing  next that $\cL \, F := Kb^{ij}\, \nabla_j\nabla_j F=Kb^{ij}\, h_{ij}\, \vec{n}= 2K \, \vec{n}$, we compute $\cL \,\eta$ on the support of $\eta$ as follows:
\begin{align*}
\cL\,\eta=\cL\, |F|^2 = 2 \langle F ,  \cL F  \rangle + 2\langle \nabla F, \nabla F \rangle_{\cL}= 4K \langle F,\vec{n}\rangle+ 2Kb^{ij}g_{ij}=4 K\langle F,\vec{n}\rangle+ 2H.
\end{align*}
Thus, $4K = 8\pi \cA^{-1}\langle \vec{e}_3,\vec{n} \rangle \leq 8\pi \cA^{-1}$ and $2H \geq 2\Lambda$ imply
\begin{align*}\label{eq:OP L eta}
\cL\,\eta\geq -8\pi \cA^{-1} |F|+ 2\Lambda.\tag{3.3}
\end{align*}

Since $w$ attains its maximum at $p_0$, we have $\nabla w (p_0)=0$, and thus 
\eqref{eq:OP Gradient} gives
\begin{align*}
\frac{\|\nabla h_{11}\|_{\cL}^2}{(h_{11})^2}(p_0)=\frac{\|\nabla \eta\|_{\cL}^2}{\eta^2}(p_0).
\end{align*}
Hence, combining $\cL w (p_0)\leq 0$, \eqref{eq:OP 2nd order derivative}, \eqref{eq:OP L eta} and the equation above yields  the following at $p_0$
\begin{align*}\label{eq:OP 1st reduced eq at p_0}
0  \geq  \frac{\cL\, h_{11}}{h_{11}}+\frac{2\Lambda}{\eta}-\frac{8\pi|F|}{\cA\eta}-\frac{2\|\nabla h_{11}\|_{\cL}^2}{(h_{11})^2}\;.\tag{3.4}
\end{align*}

To compute $\cL \, h_{11}$, we begin by differentiating $K$,
\begin{align*}\label{eq:OP gradient of K}
\nabla_1 K = Kb^{ij}\nabla_1 h_{ij}. \tag{3.5}
\end{align*}
By differentiating the equation above again, we obtain
\begin{align*}\label{eq:OP 2nd order derivative of K}
\nabla_1 \nabla_1 K=Kb^{ij}\nabla_1\nabla_1 h_{ij}+Kb^{ij}b^{kl}\nabla_1h_{ij}\nabla_1h_{kl}-Kb^{ik}b^{jl}\nabla_1h_{ij}\nabla_1h_{kl}. \tag{3.6}
\end{align*}
We can derive $\cL \, h_{11}$ from the first term $Kb^{ij}\nabla_1\nabla_1 h_{ij}$ as follows
\begin{align*}\label{eq:OP L h_11}
Kb^{ij}\nabla_1\nabla_1 h_{ij}&=Kb^{ij}\nabla_1\nabla_i h_{j1}=Kb^{ij}(\nabla_i\nabla_1 h_{j1}+R_{1ijk}h^k_1+R_{1i1k}h^k_j) \tag{3.7}\\
&=Kb^{ij}\nabla_i\nabla_j h_{11}+Kb^{ij}(h_{1j}h_{ik}-h_{1k}h_{ij})h^k_1+Kb^{ij}(h_{11}h_{ik}-h_{1k}h_{i1})h^k_j\\
&=\cL \, h_{11}-2Kh_{1k}h^k_1+KHh_{11}. 
\end{align*}
On the other hand, differentiating \eqref{eq:INT GCF} yields
\begin{align*}\label{eq:OP derivative of GCF}
\nabla_1 K= \frac{2\pi}{\cA}\, \langle \nabla_1 \vec{n},\vec{e}_3 \rangle=  -\frac{2\pi}{\cA} h_{1k} \, \langle F^k ,\vec{e}_3 \, \rangle. \tag{3.8}
\end{align*}
To get the right hand side of \eqref{eq:OP 2nd order derivative of K}, we differentiate the equation above,
\begin{align*}\label{eq:OP 2nd order derivative of GCF}
\nabla_1\nabla_1 K=  -\frac{2\pi}{\cA}\nabla_1 h_{1k}\langle F^k ,\vec{e}_3 \rangle-\frac{2\pi}{\cA} h_{1k}h^k_1\, \langle \vec{n} ,\vec{e}_3 \rangle=-\frac{2\pi}{\cA} \, \nabla_k h_{11}\langle F^k ,\vec{e}_3 \rangle -K h_{1k} h^k_1. \tag{3.9}
\end{align*}
Combining  \eqref{eq:OP 2nd order derivative of K}, \eqref{eq:OP L h_11}, and \eqref{eq:OP 2nd order derivative of GCF},  
we obtain the following at $p_0$
\begin{align*}\label{eq:OP L h_11 at p_0}
\cL\, h_{11} = & 2|\nabla_2 h_{11}|^2-2\nabla_1 h_{11}\nabla_1 h_{22}- \frac{2\pi}{\cA}\nabla_k h_{11}\langle F^k,\vec{e}_3\rangle-K^2.
\end{align*}
Hence, at $p_0$, applying the equation above to \eqref{eq:OP 1st reduced eq at p_0} and the definition of the norm $\|\cdot\|_{\cL}^2$ yield
\begin{align*}
0 \geq & \, \frac{1}{h_{11}}\big(  2|\nabla_2 h_{11}|^2-2\nabla_1 h_{11}\nabla_1 h_{22}- \frac{2\pi}{\cA}\nabla_k h_{11}\langle F^k,\vec{e}_3\rangle-K^2 \big)\\
&-\frac{2h_{22}|\nabla_1 h_{11}|^2+2h_{11}|\nabla_2 h_{11}|^2}{(h_{11})^2}+\frac{2\Lambda}{\eta}-\frac{8\pi|F|}{\cA\eta}\\
=&-\frac{2\nabla_1 h_{11}(h_{22}\nabla_1 h_{11}+h_{11}\nabla_1 h_{22})}{(h_{11})^2}+ \frac{1}{h_{11}}\big(- \frac{2\pi}{\cA}\nabla_k h_{11}\langle F^k,\vec{e}_3\rangle-K^2 \big)+\frac{2\Lambda}{\eta}-\frac{8\pi|F|}{\cA\eta}\;.
\end{align*}
However, \eqref{eq:OP gradient of K} and \eqref{eq:OP derivative of GCF} imply the following at $p_0$
\begin{align*}
h_{22}\nabla_1 h_{11}+h_{11}\nabla_1 h_{22}=\nabla_1 K = -\frac{2\pi}{\cA} h_{11}\langle F^1 ,\vec{e}_3 \rangle.
\end{align*}
Therefore, the last inequality can be reduced to
%\begin{align}
%0 \geq  - \frac{2\pi}{\cA} \frac{\nabla_2 h_{11}}{h_{11}}\langle F^2,\vec{e}_3\rangle-\frac{K^2}{h_{11}}+\frac{2\Lambda}{\eta}-\frac{8\pi|F|}{\cA\eta} \;.
%\end{align}
%{\bf I think it should be:}
\begin{align*}
0 \geq   - \frac{2\pi}{\cA} \frac{\nabla_2 h_{11} }{h_{11}}\langle F^2,\vec{e}_3\rangle
+\frac{2\pi}{\cA} \frac{\nabla_1 h_{11} }{h_{11}}\langle F^1,\vec{e}_3\rangle-\frac{K^2}{h_{11}}+\frac{2\Lambda}{\eta}-\frac{8\pi|F|}{\cA\eta} \;.
\end{align*}
Observing $(h_{11})^{-1}\nabla_i h_{11}(p_0)=-\eta^{-1}\nabla_i \eta (p_0)$ by \eqref{eq:OP Gradient} and $\nabla w(p_0)=0$, we have 
\begin{align*}
0 \geq  \frac{2\pi}{\cA}\frac{ \nabla_2 \eta}{ \eta}\langle F^2,\vec{e}_3\rangle-\frac{2\pi}{\cA}\frac{ \nabla_1 \eta}{ \eta}\langle F^1,\vec{e}_3\rangle-\frac{ K^2}{h_{11}}+\frac{2\Lambda}{\eta}-\frac{8\pi|F|}{\cA\eta}.
\end{align*}
Applying  $h_{11}(p_0)=\Lambda(p_0)$ and  \eqref{eq:INT GCF} to the inequality above, we obtain 
\begin{align*}
0 \geq  \frac{4\pi}{\cA\eta}  \langle F_2 ,F \rangle \langle F^2,\vec{e}_3\rangle-\frac{4\pi}{\cA\eta}  \langle F_1 ,F \rangle \langle F^1,\vec{e}_3\rangle-\frac{4\pi^2|\langle \vec{n},\vec{e}_3\rangle|^2}{\cA^2\Lambda}+\frac{2\Lambda}{\eta}-\frac{8\pi |F|}{\cA\eta}\;.
\end{align*}
We next multiply by $\eta \Lambda$  the last inequality and apply $| \langle F_i ,F \rangle \langle F^i,\vec{e}_3\rangle |
\leq |F|$ and $ \langle \vec{n},\vec{e}_3 \rangle \leq 1$. Then, by also  using the definition $\eta \eqqcolon (|F|^2-R^2)$, we obtain
\begin{align*}
0 \geq  -\frac{16\pi}{\cA}|F|\Lambda-\frac{4\pi^2\eta}{\cA^2}+2\Lambda^2 \geq 2\Lambda^2-\frac{16\pi}{\cA}|F|\Lambda-\frac{4\pi^2}{\cA^2}|F|^2\;.
\end{align*}
Solving the quadratic inequality of $\Lambda$, we obtain an upper bound of $\Lambda$ at $p_0$,
\begin{align*}
\Lambda \leq \frac{(4+3\sqrt{2})\pi |F|}{\cA} \leq \frac{9\pi}{\cA}|F|.
\end{align*}
Therefore, multiplying by $\eta \leq |F|^2$ yields the desired result,
\begin{align*}
\eta\Lambda(p) \leq \eta\Lambda(p_0)  \leq \frac{9\pi}{\cA}|F|^3 (p_0) \leq \frac{9\pi}{\cA}\sup_{F(q)\in \Sigma_c} |F|^{3}(q).
\end{align*}
\end{proof}

\begin{lemma}\label{lemma:OP Inscribed radius}
Suppose that $\Sigma$ is a smooth strictly convex complete solution of \eqref{eq:INT GCF}, and the height $\langle F(p),\vec{e_3} \rangle$ attains this minimum at a point $p_0 \in M^2$. Given any point $p \in M^2$ with $|F(p)-F(p_0)| \gg \text{diam}(\Omega)$, there exists a ball $B_r(Z)$ such that the convex hull of $\Sigma$ contains $B_r(Z)$, $F(p) \in \bd B_r(Z) $, and \[r \geq   10^{-3}\mathcal{A}|F(p)-F(p_0)|^{-1} .\]
\end{lemma}

\begin{proof}
We set $R=2\text{diam}(\Omega)$ and $Y=F(p_0)+2\langle F(p)-F(p_0),e_2\rangle $, and apply Theorem \ref{thm:OP Optimal regularity}. Then, we have $H \leq 10^3\mathcal{A}^{-1}|F(p)-F(p_0)|$ in the ball $B_1(F(p))$. Hence, we can put a ball $B_r(Z)$ satisfying the desired conditions.
\end{proof}

\begin{lemma}\label{lemmma: OP Height bound by level set}
Given a point $\vec{x}_0 \in \Omega$ with $Du(\vec{x_0})\neq 0$, the height $u(\vec{x}_0)-\inf_\Omega u$  is bounded by $\text{diam}(\Omega)$, an upper estimate for $\cA(\Omega)$, and a lower estimate for the area of $\Omega^{\vec{x}_0}=\{\vec{x}\in \Omega: \langle \vec{x}-\vec{x}_0,Du(\vec{x}_0)\rangle \geq 0\}$.
\end{lemma}

\begin{proof}
We define $\overline \Omega^{\vec{x}_0}=\{\vec{x}\in\Omega:u(\vec{x}) \geq u(\vec{x}_0)\}$ and $\Gamma_\Omega^{\vec{x}_0}=\{\vec{x}\in\Omega:u(\vec{x}) = u(\vec{x}_0)\}$. Since $\Gamma_\Omega^{\vec{x}_0}$ is a convex curve, $ \Omega^{\vec{x}_0} \subset \overline \Omega^{\vec{x}_0}$ holds, namely $\cA(\Omega^{\vec{x}_0}) \leq \cA( \overline \Omega^{\vec{x}_0})$.

We define the normal map $\mathcal{N}:\Omega \to S^2$ by $\mathcal{N}=(-Du,1)(1+|Du|^2)^{-\frac{1}{2}}$. Then, by \eqref{eq:INT GCF} we have 
\begin{align*}
|\mathcal{N}(\overline \Omega^{\vec{x}_0})|=2\pi  \mathcal{A}(\overline \Omega^{\vec{x}_0})/\mathcal{A}(\Omega) \geq 2\pi  \mathcal{A}(\Omega^{\vec{x}_0})/\mathcal{A}(\Omega). 
\end{align*}
Therefore, the ratio of a lower estimate for $\mathcal{A}(\Omega^{\vec{x}_0})$ and an upper estimate for $\mathcal{A}(\Omega)$ yields an upper estimate for $D_{\vec{x}_0}=\inf\{|Du|(\vec{x}):\vec{x}\in {\Gamma_\Omega^{\vec{x}_0}}\}$. Therefore, the convexity of $u$ guarantees \[u(\vec{x}_0)-\inf_{\Omega} u \leq D_{\vec{x}_0} \cdot \text{diam}(\Omega).\] 
\end{proof}

We close this section with the following result which will be used in the proof of Theorem \ref{thm:DTF main thm}. 

\begin{theorem}\label{thm:OP Graph in region}
Suppose that a strictly convex and smooth solution $u(x,y)$ to \eqref{eq:INT GCFeq} satisfies $\frac{Du}{|Du|}=e_1$ at a point $(x_0,y_0) \in \Omega$. Then, given  a point $(x_0,\bar y)\in \bd\Omega$, the distance $|y_0-\bar y|$ is bounded below by some constant depending on $u(x_0,y_0)-\inf_\Omega u$ and $\text{diam}(\Omega)$. 
\end{theorem}

\begin{proof}
Lemma \ref{lemma:OP Inscribed radius} gives a ball $B_r(Z)$ such that $(x_0,y_0,u(x_0,y_0))\in \bd B_r(Z)$ holds and the convex hull of $\Sigma$ contains $B_r(Z)$. Moreover, the radius $r$ is bounded below by some constant depending on $u(x_0,y_0)-\inf_\Omega u$ and $\text{diam}(\Omega)$.  Since $\frac{Du}{|Du|}=e_1 $ holds at $(x_0,y_0)$, we have $\langle Z,e_2 \rangle=y_0$. Thus, the points $Z\pm r e_2 \in \bd B_r(Z)$ are contained in the convex hull of $\Sigma$. This yields the desired result.
\end{proof}

\section{Partial derivative bound}

%Assume that  $\Omega$ is an open  bounded strictly convex and  smooth domain of\, $\rtw$  and assume
%in addition that $\Omega$  is axially symmetric. Let $u$ be the unique solution of \eqref{eq:INT GCFeq}  on $\Omega$
%which defines the surface $\Sigma$.  
%The following simple property readily  follows from the uniqueness of solutions.
%
%\begin{proposition}[Symmetry of solutions]\label{prop:PGE Solutions with sym}
%Let $\Omega$ be an open  bounded strictly convex and  smooth  subset of\, $\rtw$ which is axially symmetric. 
%Then, a solution $u$ of \eqref{eq:INT GCFeq} also is axially  symmetric.
%\end{proposition}
%
%
%The symmetry of $u$ implies that  the half surface $\{(x,y,z)\in \Sigma: y \leq 0\}$  is the graph of a function $h:\Omega_y \to \mathbb{R}$, that is

Theorem \ref{thm:OP Graph in region} in the previous section implies that a portion of the boundary $\Omega$ is close enough to a straight segment $L=\{(x,b):a_1 \leq x \leq a_2\}$, then the gradient $Du$ can not parallel to $e_1$ in a neighborhood of $L_\epsilon=\{(x,b): a_1+\epsilon \leq x \leq a_2-\epsilon\}$. Namely, the surface $\Sigma$ can be considered as a graph with respect to the height vector $e_2$ a neighborhood of $L_\epsilon \times \mathbb{R}$. We will discuss this idea rigorously in the proof of the main theorem \ref{thm:DTF main thm}.

\medskip

In this section, we consider a domain $\Omega_y=\{(x,z):(x,y,z)\in \Sigma\}$ where a convex function $h:\Omega_y\to \mathbb{R}$ is defined by $(x,h(x,z),z) \in \Sigma$. Moreover, we may assume that the graph of $h(x,z)$ contains the portion $\{(x,y,z)\in \Sigma:-1\leq x\leq 1, -1\leq y\leq 0\}$ of $\Sigma$.
 
\medskip 
 
 We begin by deriving the equation of the function $h$.
\begin{equation}\label{eq:PGE GCFeq y-co}
\frac{\det D^2 h}{(1+|Dh|^2)^\frac{3}{2}}=K \, \big(1+|Dh|^2\big)^{\frac{1}{2}}=\frac{2\pi}{\cA(\Omega)}\langle \vec{e}_3, \vec{n} \rangle \big(1+|Dh|^2\big)^{\frac{1}{2}}=-\frac{2\pi}{\cA(\Omega)}h_z. \tag{4.1}
\end{equation}
The right hand side of the equation above can written as  $(2\pi/\cA) h_v \, \langle  -e_2   ,v \rangle$, where $v$ is the outward normal direction of the level set of $h$. Thus, the degenerate Monge-Ampere equation \eqref{eq:PGE GCFeq y-co} has two degenerating factors $h_v = |Dh|$ and $\langle -e_2 ,v \rangle$. In this section, we study the lower bound for $\langle -e_2 ,v \rangle$ which corresponds to the upper bound of $|\bd_x u|$,  the {\em partial derivative bound}. Notice that $|\bd_x u|$ is bounded even on the flat sides.

\medskip

To obtain  the lower bound for $\langle -e_2,v \rangle$, we will construct an one parameter family of  very wide but short supersolutions  
$\varphi_\alpha$ of equation \eqref{eq:PGE GCFeq y-co} with $\cA(\Omega)<\cR$ for some constant $\cR $. We will then cut  the graph of $\varphi_\alpha$ so that each of them each   contained in a narrow cylinder. 
By  sliding  $\varphi_\alpha$   along the $z$-axis we will   estimate the partial derivative $\bd_x u$ at a touching point which will lead to 
the  desired upper bound for $\bd_x u$ as stated in  Theorem \ref{thm:PGE Gradient bound for level set}.

\def\Aangle{87}
\def\Bangle{0}
\def\Cangle{0}
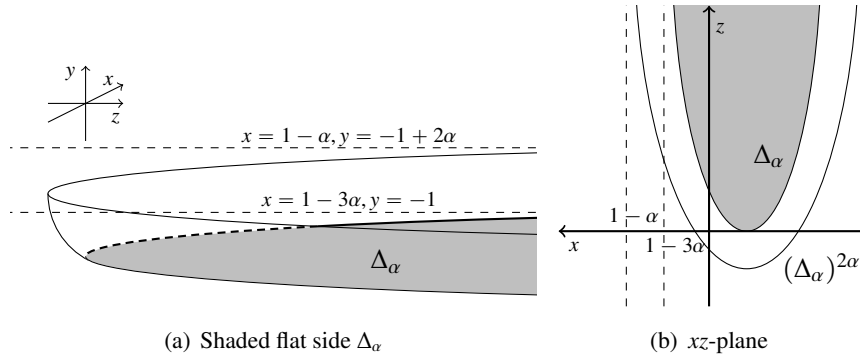
\begin{figure}[h]
\subfigure[Shaded flat side $\Delta_{\alpha}$]{
\begin{tikzpicture}\label{fig: Graph of supersolution}
[x=1cm,y=1cm] \clip(-12,-1.5) rectangle (-5,2.5);
\tikzset{Bplane/.estyle={cm={
cos(\Bangle),sin(\Bangle)*sin(\Aangle),sin(\Cangle)*sin(\Bangle),
cos(\Cangle)*cos(\Aangle)-sin(\Cangle)*cos(\Bangle)*sin(\Aangle),(0,-0.835)}}}
\tikzset{Tplane/.estyle={cm={
cos(\Bangle),sin(\Bangle)*sin(\Aangle),sin(\Cangle)*sin(\Bangle),
cos(\Cangle)*cos(\Aangle)-sin(\Cangle)*cos(\Bangle)*sin(\Aangle),(0,0)}}}
\draw (-11.5,0) arc (180:243:1cm);
\draw[Bplane,color=white,fill=lightgray] (0, 0) circle (11cm);
\draw[Tplane] circle (11.5cm);
\draw[Bplane,thick] (0, 11) arc (90:136:11cm);
\draw[Bplane,thick,densely dashed] (11*cos 136, 11*sin 136) arc (136:180:11cm);
\draw[Bplane] (-11, 0) arc (180:270:11cm);
\draw[->] (-11.5,1.2) -- (-10.5,1.2);
\draw[->] (-11.5,0.95) -- (-10.5,1.45);
\draw[->] (-11,0.7) -- (-11,1.7);
\draw[dashed] (-12.5,-0.25) -- (-2,-0.25);
\draw[dashed] (-12.5,0.61) -- (-2,0.61);
\begin{scriptsize}
\draw[color=black] (-7,-0.9) node[scale=1.3] {$\Delta_{\alpha}$};
\draw[color=black] (-7.5,-0.1) node {$x=1-3\alpha, y=-1$};
\draw[color=black] (-7.5,0.75) node {$x=1-\alpha, y=-1+2\alpha$};
\draw[color=black] (-10.7,1.5) node {$x$};
\draw[color=black] (-11.2,1.6) node {$y$};
\draw[color=black] (-10.6,1) node {$z$};
\end{scriptsize}
\end{tikzpicture}
}
\subfigure[$xz$-plane]{
\begin{tikzpicture}\label{fig: xz plane supersolution}
[x=1cm,y=1cm] \clip(-2.5,-5) rectangle (1.5,-1);
\draw[fill=lightgray] (0,0) ellipse (1cm and 4cm);
\draw(0,0) ellipse (1.5cm and 4.5cm);
\draw[thick,->] (2.5,-4)--(-2.5,-4);
\draw[thick,->] (-0.5,-5) -- (-0.5,-1);
\draw[dashed] (-1.1,-5) -- (-1.1,-0.5);
\draw[dashed] (-1.6,-5) -- (-1.6,-0.5);
\begin{scriptsize}
\draw[color=black] (0.3,-3) node[scale=1.3] {$\Delta_{\alpha}$};
\draw[color=black] (1,-4.5) node[scale=1.3] {$(\Delta_\alpha)^{2\alpha} $};
\draw[color=black] (-0.95,-4.2) node {$1-3\alpha$};
\draw[color=black] (-1.5,-3.8) node {$1-\alpha$};
\draw[color=black] (-2.3,-4.2) node {$x$};
\draw[color=black] (-0.35,-1.3) node {$z$};
\end{scriptsize}
\end{tikzpicture}
}
\caption{Supersolution $\varphi_\alpha$}\label{fig: Supersolution}
\end{figure}

\begin{definition}[Barrier construction]\label{def:PGE Flat Barrier}
Given a constant $\alpha \in (0,1/6)$, denote by  $\Delta_\alpha $ the convex set 

\begin{equation*}
\Delta_\alpha = \Bigg\{(x,z) \in \rtw: x \in \Big(-\frac{\cR}{4\alpha}+1-3\alpha,1-3\alpha\Big) ,\, z \geq -\frac{\cR}{4\alpha \pi} \log \cos \bigg(\frac{4\alpha \pi }{\cR}\Big(x-1+3\alpha+\frac{\cR}{8\alpha}\Big)\bigg) \Bigg\}.
\end{equation*}

We denote by $d_{\Delta_\alpha}(x,z)$ the distance function $d((x,z),\Delta_\alpha)$. In particular, if $(x,z) \in \Delta_\alpha$, then $d_{\Delta_\alpha}(x,z)=0$. By using $d_{\Delta_\alpha}(x,z)$, we define the $2\alpha$-extension $(\Delta_\alpha)^{2\alpha} $ of $\Delta_\alpha $ by
\begin{align*}
(\Delta_\alpha)^{2\alpha } = \{ (x,z) \in \rtw : d_{\Delta_\alpha}(x,z)\leq 2\alpha\}.
\end{align*}
Finally,   we define the  function $\varphi_\alpha:\text{cl}\big((\Delta_\alpha)^{2\alpha }\setminus \Delta_\alpha) \to \mathbb{R}$ by
\begin{equation*}
\varphi_\alpha(x,z) =-1+2\alpha - \sqrt{4\alpha^2 -d^2(\Delta_\alpha)(x,z)} \, .
\end{equation*}
This is all shown in Figure 3. 

\end{definition}

\begin{lemma}[Supersolution]\label{lemma:PGE Supersolution}
Given a constant $\alpha \in (0,1/6)$, the function $\varphi_\alpha$ in \rm{Definition} \ref{def:PGE Flat Barrier} is a convex function satisfying
\begin{equation*}
\frac{\det D^2\varphi}{(1+|D\varphi|^2)^{\frac{3}{2}}}\leq -\frac{2\pi}{\cR}\varphi_z \, .
\end{equation*}
\end{lemma}

\begin{proof}
For convenience, we let $\varphi$ and $d$ denote $\varphi_\alpha$ and $d_{\Delta_\alpha}$  respectively. For each point $p\in \rtw$ with $d(p)>0$, we denote by $\tau (p)$ and $v(p)$ the tangential and the normal direction of a level set  $L_{d(p)}(d)$ of the distance function $d$ satisfying $\langle \tau ,e_1 \rangle  \geq 0$ and $\langle v ,e_2 \rangle  \leq 0$,  respectively. Then, we have
\begin{align*}\label{eq:PGE Dd}
& Dd=v, && d_v =|Dd|=1, && d_\tau=0.\tag{4.2}
\end{align*}
We observe $v(p)=v(p+\epsilon v(p))$ for all $\epsilon \in \mathbb{R}$ with $d(p+\epsilon v(p))>0$, which implies
\begin{align*}\label{eq:PGE d_vv}
d_{vv}=d_{v\tau}=0. \tag{4.3}
\end{align*}
To derive $d_{\tau\tau}(p)$, given a point $p_0$, we consider the immersion $\gamma:\mathbb{R}\to \rtw$ satisfying $d(\gamma(s))=d(p_0)$ with $\gamma(0)=p_0$, where $s$ is the arc length parameter of the level set, $\gamma(\mathbb{R})=L_{d(p_0)}(d)$. By differentiating $d(\gamma(s))=d(p_0)$ twice with respect to $s$, we obtain
\begin{align*}
\langle \gamma_s,(D^2 d) \gamma_s \rangle +\langle Dd, \gamma_{ss} \rangle =0 
\end{align*}
We can observe that $\gamma_s(0)=\tau(p_0)$ and $\gamma_{ss}(0)=-\kappa(p_0)v(p_0)$, where $\kappa(p)>0$ is the curvature of $L_{d(p)}(d)$ at $p$. Hence, $Dd(\gamma(0))=v(p_0)$ implies $\langle \tau,(D^2 d) \tau \rangle +\langle -\kappa v, v \rangle =0$ at $p_0$. Thus,
\begin{align*}\label{eq:PGE d_tt}
d_{\tau\tau}(p)=\kappa(p). \tag{4.4}
\end{align*}
Hence we  can directly derive from \eqref{eq:PGE Dd}. \eqref{eq:PGE d_vv}  and \eqref{eq:PGE d_tt} the following 
holding at each $p \in \text{cl}\big((\Delta_\alpha)^{2\alpha}\setminus \Delta_\alpha)$
\begin{align*}
& \varphi_v = d(4\alpha^2-d^2)^{-\frac{1}{2}}, && \varphi_\tau=0, &&\varphi_{vv} = 4\alpha^2(4\alpha^2-d^2)^{-\frac{3}{2}} , && \varphi_{v\tau}=0, && \varphi_{\tau\tau}=\kappa \,  \varphi_v . 
\end{align*}
Therefore, $\varphi$ is a convex function.

Next, combining the equalities above yields
\begin{align*}\label{eq:PGE varphi eq}
\frac{\det D^2 \varphi}{(1+|D\varphi|^2)^{\frac{3}{2}}}=\frac{\varphi_{vv}\varphi_{\tau\tau}}{(1+\varphi_v^2)^{\frac{3}{2}}}=\frac{1}{2\alpha}\varphi_{\tau\tau}=\frac{1}{2\alpha} \, \kappa \,  \varphi_v .\tag{4.5}
\end{align*}
Now, we consider the point $p_0=p-d(p) v(p) \in \bd\Delta_\alpha$. Then, the convexity of $\bd\Delta_\alpha$ leads to
\begin{align*}\label{eq:PGE curvature comparison}
\kappa(p) \leq  \kappa(p_0).\tag{4.6}
\end{align*}
We recall that the Grim Reaper curve $\bd\Delta_\alpha$ is the graph of the convex function $f_\alpha(x)$ defined by

\begin{align*}\label{eq:PGE Grim reaper curve f_a}
f_\alpha(x)=-\frac{\cR}{4\pi\alpha} \log \cos \bigg(\frac{4\alpha \pi }{\cR}\Big(x-1+3\alpha+\frac{\cR}{8\alpha}\Big)\bigg).\tag{4.7}
\end{align*}

Hence, at $x_0$ with $p_0=(x_0,f_\alpha(x_0))$, the following holds

\begin{align*}
\kappa(p_0)=\frac{f_\alpha''(x_0)}{(1+|f_\alpha'(x_0)|^2)^{\frac{3}{2}}}=\frac{4\pi \alpha /\cR}{(1+|f_\alpha'(x_0)|^2)^{\frac{1}{2}}}=-\frac{4\pi \alpha}{\cR} \langle v(p_0),e_2 \rangle. 
\end{align*}

Thus, $v(p_0)=v(p-d(p) v(p))=v(p)$  implies

\begin{align*}\label{eq:PGE Level set inequality}
\kappa(p)\leq \kappa(p_0)=-\frac{4\pi \alpha}{\cR} \langle v(p_0),e_2 \rangle=-\frac{4\pi \alpha}{\cR} \langle v(p),e_2 \rangle.\tag{4.8}
\end{align*} 

Therefore,  given a point $p \in \text{cl}\big((\Delta_\alpha)^{2\alpha}\setminus \Delta_\alpha)$, \eqref{eq:PGE varphi eq}, \eqref{eq:PGE curvature comparison}, and \eqref{eq:PGE Level set inequality} give the desired result
\begin{align*}
\frac{\det D^2 \varphi}{(1+|D\varphi|^2)^{\frac{3}{2}}}=\frac{1}{2\alpha }\varphi_v \kappa \leq -\frac{2\pi}{\cR}\varphi_v \langle v,e_2 \rangle=-\frac{2\pi}{\cR} \, \big(\varphi_v \langle v,e_2 \rangle+\varphi_\tau \langle \tau,e_2 \rangle \big)=-\frac{2\pi}{\cR}\varphi_z \,.
\end{align*}
\end{proof}

\begin{theorem}[Partial derivative bound]\label{thm:PGE Gradient bound for level set}
Let $\Omega$ be an open  strictly convex smooth  subset of\, $\rtw$ such that $\mathcal{A}(\Omega) <\cR$ for some constant $\cR $ and $[-1,1]\times[-1,0]\subset \Omega$. In addition, a corresponding solution $\Sigma$ to \eqref{eq:INT GCF} is a convex graph with respect to the height vector $e_2$ in $[-1,1]\times [-1,0]\times \mathbb{R}$. Then, the solution $u :\Omega  \to \mathbb{R}$ to \eqref{eq:INT GCFeq} defining $\Sigma$  satisfies
\begin{align*}
\bd_x u  (1-5\alpha,-1+\epsilon) \leq 2+3\cot \frac{4\pi\alpha^2}{\cR},
\end{align*}
for all $\epsilon \in (0,1/2)$ and $ \alpha \in (0,1/6)$.
\end{theorem}

\begin{proof}
%Since the  domain $\Omega$ has   strictly convex and smooth boundary,  Theorem \ref{thm:INT Urbas} implies that the graph $\Sigma=\{(x,y,u(x,y)):(x,y)\in \Omega\}$ of the function $u$ is a strictly convex complete smooth solution  of \eqref{eq:INT GCF}. In addition  by Proposition \ref{prop:PGE Solutions with sym} the function $u$ is axially symmetric
%and  we can define  
%a convex set $\Omega_y$ and a convex function $h:\Omega_y \to \mathbb{R}$ by
%\begin{align*}
%&\Omega_y =\{(x,z):(x,y,z)\in \Sigma\}, &&
%\{(x,y,z)\in\Sigma: y\leq 0\}  =\{ (x,h(x,z),z):(x,z)\in \Omega_y\}.
%\end{align*}
%Then, $(x,y,u(x,y))=(x,h(x,z),z)$ implies that the function $h$ satisfies \eqref{eq:PGE GCFeq y-co}. Thus, by the given condition $\cA(\Omega) \leq (8/3)^2<8$ and Lemma \ref{lemma:PGE Supersolution}, the function $\varphi_\alpha$ in Definition \ref{def:PGE Flat Barrier} is a supersolution for \eqref{eq:PGE GCFeq y-co}. 

To construct a barrier $\Phi_{\epsilon,\alpha}^{t_{\epsilon,\alpha}}$, we will cut the graph of $\epsilon+\varphi_\alpha$ by $\{1-4\alpha\} \times \rtw$ (the blue section in Figure \ref{fig:PGE Barrier Construction}) and slide it along $z$-direction until it touches $\Sigma$ at a point $P_0$. We will show that the contact point $P_0$ is contained in $\{1-4\alpha\} \times \rtw$, namely $P_0$ is a point on the front part of the boundary $\bd\Phi_{\epsilon,\alpha}^{t_{\epsilon,\alpha}}$ of the barrier. (See the blue curve $\Gamma_F$ in Figure \ref{fig:PGE Barrier Construction}). Then, we will estimate the partial derivative $\bd_x u$ at $P_0$ by comparing with the barrier $\Phi_{\epsilon,\alpha}^{t_{\epsilon,\alpha}}$ at $P_0$. After obtaining the  bound on $\bd_x u$ at $P_0$, we will use the convexity of the solution $\Sigma$ the barrier $\Phi_{\epsilon,\alpha}^{t_{\epsilon,\alpha}}$ to show the desired bound of $\bd_x u$ at $(1-5\alpha,-1+\epsilon)$.

\def\Aangle{87}
\def\Bangle{0}
\def\Cangle{0}
\begin{figure}[h]
\subfigure[Blue section cutting the supersolution $\varphi_\alpha$]{
\begin{tikzpicture}\label{fig:PGE Barrier Construction}
[x=1cm,y=1cm] \clip(-12,-2) rectangle (-6,2);
\shade[top color=blue!30!,bottom color=blue!10!] (-12,2.57+3* sin 230) rectangle (-6,2.57+3* sin 270);
\draw[thick,densely dashed] (-11.5,0) arc (180:243:1cm);
\tikzset{Bplane/.estyle={cm={
cos(\Bangle),sin(\Bangle)*sin(\Aangle),sin(\Cangle)*sin(\Bangle),
cos(\Cangle)*cos(\Aangle)-sin(\Cangle)*cos(\Bangle)*sin(\Aangle),(0,-0.835)}}}
\tikzset{Tplane/.estyle={cm={
cos(\Bangle),sin(\Bangle)*sin(\Aangle),sin(\Cangle)*sin(\Bangle),
cos(\Cangle)*cos(\Aangle)-sin(\Cangle)*cos(\Bangle)*sin(\Aangle),(0,0)}}}
\draw[Bplane,color=white,fill=lightgray] (0, 0) circle (11cm);
\tikzset{Mplane/.estyle={cm={
cos(\Bangle),sin(\Bangle)*sin(\Aangle),sin(\Cangle)*sin(\Bangle),
cos(\Cangle)*cos(\Aangle)-sin(\Cangle)*cos(\Bangle)*sin(\Aangle),(0,-0.3)}}}
\draw[Bplane,color=white,fill=lightgray] (0, 0) circle (11cm);
\draw[line width=0.04cm,color=blue] (-8,-0.43) arc (270:230:3cm);
\draw[Bplane,line width=0.04cm] (0,11) arc (90:137:11cm);
\draw[Mplane,thick,color=red] (0,11) arc (90:150:11cm);
\draw[Bplane,thick,densely dashed] (11*cos 137, 11*sin 137) arc (137:270:11cm);
\draw[Tplane,thick] (0,11) arc (90:150:11.5cm);
\draw[Tplane,thick,densely dashed] (11.5*cos 150, 11.5*sin 150) arc (150:200:11.5cm);
\draw[->] (-11.5,1.2) -- (-10.5,1.2);
\draw[->] (-11.5,0.95) -- (-10.5,1.45);
\draw[->] (-11,0.7) -- (-11,1.7);
\draw[color=red,fill=red] (-8+3*cos 240,2.57+3* sin 240) circle (0.06cm);
\draw[fill=black] (-8+3*cos 270,2.57+3* sin 270) circle (0.06cm);
\draw[fill=black] (-8+3*cos 230,2.57+3* sin 230) circle (0.06cm);
%\draw[fill=black] (-9.3,-0.53) circle (0.06cm);
\draw[color=blue,dashed] (-12,2.57+3* sin 230) -- (-6,2.57+3* sin 230);
\draw[color=blue,dashed] (-12,2.57+3* sin 270) -- (-6,2.57+3* sin 270);
\begin{scriptsize}
\draw[color=black] (-10,0.5) node[scale=1.3] {$P_1$};
%\draw[color=black] (-9.3,-0.85) node[scale=1.3] {$P_3$};
\draw[color=black] (-9.8,-0.1) node[scale=1.3] {$P_0$};
\draw[color=black] (-8,-0.7) node[scale=1.3] {$P_2$};
\draw[color=black] (-7,0.7) node[scale=1.3] {$\Gamma_T$};
\draw[color=black] (-6.5,-0.15) node[scale=1.3] {$\Gamma_B$};
\draw[color=blue] (-8.5,-0.2) node[scale=1.3] {$\Gamma_F$};
\draw[color=red] (-7.5,-0.1) node {$y=y_0$};
\draw[color=blue] (-11.2,0.4) node {$x=1-4\alpha$};
\draw[color=black] (-10.7,1.5) node {$x$};
\draw[color=black] (-11.2,1.6) node {$y$};
\draw[color=black] (-10.6,1) node {$z$};
\end{scriptsize}
\end{tikzpicture}
}
\subfigure[$xz$-plane]{
\begin{tikzpicture}\label{fig:PGE Barrier xz plane}
[x=1cm,y=1cm] \clip(-5,-6) rectangle (-1,-2);
\draw[fill=gray!15!] (-2.7*1.2,-3.93*1.2) -- (-2.7*1.2+2*0.0841,-3.93*1.2+2*0.0541)--(-2.7*1.2+2*0.03,-3.93*1.2+2*0.1382)--(-2.7*1.2-2*0.0541,-3.93*1.2+2*0.0841)--(-2.7*1.2,-3.93*1.2);
\draw[fill=gray!15!] (-2.25*1.2,-3.64*1.2) -- (-2.25*1.2-2*0.0841,-3.64*1.2-2*0.0541)--(-2.25*1.2-2*0.1382,-3.64*1.2+2*0.03)--(-2.25*1.2-2*0.0541,-3.64*1.2+2*0.0841)--(-2.25*1.2,-3.64*1.2);
\draw[dashed, fill=lightgray] (0,0) ellipse (3*1.2cm and 5.5*1.2cm);
\draw[->] (-1.1,-2.4-0.3)--(-2.1,-2.4-0.3);
\draw[->] (-1.6,-2.9-0.3) -- (-1.6,-1.9-0.3);
\draw[color=blue,dashed] (-2.7*1.2,-5*1.2) -- (-2.7*1.2,-0.5*1.2);
\draw[color=blue,line width=0.04cm] (-2.7*1.2,-4.55*1.2) --(-2.7*1.2,-2.35*1.2);
\draw[dashed] (0,0) ellipse (4*1.2cm and 6.2*1.2cm);
\draw[fill=black] (-2.7*1.2,-4.55*1.2) circle (0.06cm);
\draw[fill=black] (-2.7*1.2,-2.35*1.2) circle (0.06cm);
\draw[color=red,fill=red] (-2.7*1.2,-3.93*1.2) circle (0.06cm);
\draw (-2.7*1.2,-3.93*1.2) -- (-2.25*1.2,-3.64*1.2);
\draw[fill=black] (-2.25*1.2,-3.64*1.2) circle (0.06cm);
\begin{scriptsize}
\draw[color=black] (-2.1,-3.64*1.2) node[scale=1.3] {$(\bar x_0, \bar z_0)$};
\draw[color=black] (-2.7*1.2+0.3,-2.35*1.2) node[scale=1.3] {$P_2$};
\draw[color=black] (-2.7*1.2+0.3,-4.55*1.2) node[scale=1.3] {$P_1$};
\draw[color=blue] (-2.7*1.2-0.8,-4.55*1.2-0.2) node {$x=1-4\alpha$};
\draw[color=black] (-2.7*1.2+0.3,-3.93*1.2-0.2) node[scale=1.3] {$P_0$};
\draw[color=black] (-2.7*1.2-0.9,-4.55*1.2+1) node[scale=1.3] {$\Gamma_T$};
\draw[color=black] (-2.7*1.2-0.4,-2.35*1.2+0.4) node[scale=1.3] {$\Gamma_B$};
\draw[color=blue] (-2.7*1.2-0.2,-3.5) node[scale=1.3] {$\Gamma_F$};
\draw[color=red] (-2.7*1.2-0.5,-4.15) node {$ f^{\, 0}_{\epsilon,\alpha}$};
\draw[color=black] (-2,-5) node {$f_{\epsilon,\alpha}$};
\draw[color=black] (-2.05,-2.9) node {$x$};
\draw[color=black] (-1.75,-2.3) node {$z$};
\end{scriptsize}
[x=1cm,y=1cm] \clip(-5,-6) rectangle (-2.7*1.2,-2);
\draw[thick] (0,0) ellipse (4*1.2cm and 6.2*1.2cm);
\draw[thick] (0,0) ellipse (3*1.2cm and 5.5*1.2cm);
\draw[red,thick] (0,0) ellipse (3.6*1.2cm and 6*1.2cm);
\end{tikzpicture}
}
\caption{Sliding barrier}\label{fig:PGE Barrier}
\end{figure}
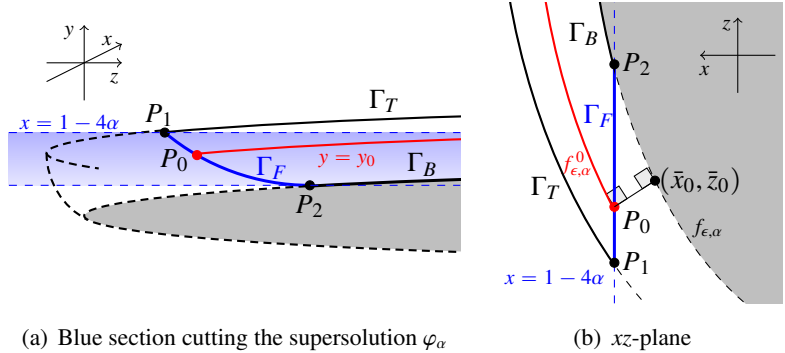

\noindent\textit{Step 1 : Sliding barrier construction.} We denote by $\Phi_\alpha$ the graph of $\varphi_\alpha$ in $[1-4\alpha,+\infty) \times \rtw$,
\begin{align*}
\Phi_{\alpha} = \big\{ (x,\varphi_\alpha(x,z),z):(x,z) \in \text{cl}\big((\Delta_\alpha)^{2\alpha}\setminus \Delta_\alpha),  \; x\geq 1-4\alpha \big\}.
\end{align*}
Then, given constants $\epsilon \in (0,1/2)$ and $t \in \mathbb{R}$, we translate $\Phi_\alpha$ by $\epsilon\vec{e}_2+t\vec{e}_3$,
\begin{align*}
\Phi_{\epsilon,\alpha}^t = \big\{ (x,y+\epsilon,z+t):(x,y,z) \in \Phi_\alpha\}.
\end{align*}
Notice that	definition of $\varphi_\alpha$ guarantees
\begin{align*}
\Phi_{\epsilon,\alpha}^t \subset [1-4\alpha,1-\alpha] \times [-1+\epsilon,-1+\epsilon+2\alpha]\times [-t-2\alpha ,+\infty).
\end{align*}
Hence, there exists a constant $t_{\epsilon,\alpha} \in \mathbb{R}$ and a point $P_0=(x_0,y_0,z_0) \in  \rth$ satisfying
\begin{align*}\label{eq:PGE Contact point}
&\Sigma  \bigcap \Phi_{\epsilon,\alpha}^t =\emptyset\quad \text{for}\; t > t_{\epsilon,\alpha}, && P_0\in \Sigma \bigcap \Phi_{\epsilon,\alpha}^{t_{\epsilon,\alpha}}. \tag{4.9}
\end{align*}

We denote by $\Delta_{\epsilon,\alpha}$ the projection of $\Phi_{\epsilon,\alpha}^{t_{\epsilon,\alpha}}$ into the $xz$-plane, and consider $\Phi_{\epsilon,\alpha}^{t_{\epsilon,\alpha}}$ as the graph of a function $\varphi_{\epsilon,\alpha}:\Delta_{\epsilon,\alpha}\to \mathbb{R}$
\begin{align*}
&\Delta_{\epsilon,\alpha}  =\{(x,z):(x,y,z)\in \Phi_{\epsilon,\alpha}^{t_{\epsilon,\alpha}}\}, &&
\Phi_{\epsilon,\alpha}^{t_{\epsilon,\alpha}}  =\{ (x,\varphi_{\epsilon,\alpha}(x,z),z):(x,z)\in \Delta_{\epsilon,\alpha}\}.
\end{align*}

\medskip

\noindent \textit{Step 2 : Position of the contact point $P_0$.} In this step, we will show that the contact point $P_0$ is contained in the front part $L^{x}_{1-4\alpha} (\Phi_{\epsilon,\alpha}^{t_{\epsilon,\alpha}})$ of the boundary $\bd\Phi_{\epsilon,\alpha}^{t_{\epsilon,\alpha}}$. 

First of all, the contact point $P_0$ can not be an interior point of $\Phi_{\epsilon,\alpha}^{t_{\epsilon,\alpha}}$, because $\varphi_{\epsilon,\alpha}$ is a supersolution. Thus, $P_0$ is a point on the boundary $\bd\Phi_{\epsilon,\alpha}^{t_{\epsilon,\alpha}}$ of $\Phi_{\epsilon,\alpha}^{t_{\epsilon,\alpha}}$. We observe that the boundary $\bd\Phi_{\epsilon,\alpha}^{t_{\epsilon,\alpha}}$ can be decomposed into the top $\Gamma_T$, bottom $\Gamma_B$, and front $\Gamma_F$ boundary as following
\begin{align*}\label{eq:PGE Boundary decomposition}
&\bd\Phi_{\epsilon,\alpha}^{t_{\epsilon,\alpha}}=\Gamma_T \bigcup \Gamma_B \bigcup \Gamma_F, && \Gamma_T=L^{y}_{-1+\epsilon+2\alpha} (\Phi_{\epsilon,\alpha}^{t_{\epsilon,\alpha}}), && \Gamma_B=  L^{y}_{-1+\epsilon} (\Phi_{\epsilon,\alpha}^{t_{\epsilon,\alpha}}), && \Gamma_F= L^{x}_{1-4\alpha} (\Phi_{\epsilon,\alpha}^{t_{\epsilon,\alpha}}).\tag{4.10}
\end{align*}
We denote by $P_1$ and $P_2$ the end point of the top $\Gamma_T$ and the bottom $\Gamma_B$ boundary, respectively
\begin{align*}\label{eq:PGE P_1 & P_2}
 P_1&=(x_1,y_1,z_1) = \Gamma_T \bigcap \Gamma_F=(1-4\alpha,-1+\epsilon+2\alpha,z_1) , \tag{4.11}\\
 P_2 &=(x_2,y_2,z_2) = \Gamma_B \bigcap \Gamma_F=(1-4\alpha,-1+\epsilon,z_2). 
\end{align*}

On the other hand \eqref{eq:PGE Contact point} gives $\text{cl}(\Delta_{\epsilon,\alpha})\eqqcolon\Delta_{\epsilon,\alpha}  \subset \text{Int}(\Omega_y)$ and $h(x,z)\leq \varphi_{\epsilon,\alpha}(x,z)$ on $\Delta_{\epsilon,\alpha}$. Also, we have $|D \varphi_{\epsilon,\alpha}|= +\infty$ on $\Gamma_T$. Hence, if $P_0 \in \big(\Gamma_T \setminus \{P_1\} \big)$, then $|Dh|=+\infty$ holds at $P_0$ by $h\leq \varphi_{\epsilon,\alpha}$, which contradicts to $\text{cl}(\Delta_{\epsilon,\alpha}) \subset \text{Int}(\Omega_y)$. Thus,
\begin{align*}
P_0 \not \in \big(\Gamma_T \setminus \{P_1\} \big)
\end{align*}
Moreover, we have $|D \varphi_{\epsilon,\alpha}|= 0$ on $\Gamma_B$. Thus, if $P_0 \in \big(\Gamma_B \setminus \{P_2\} \big)$, then $|Dh|=0$ holds at $P_0$. However, $\Sigma$ is a strictly convex complete surface, which means $|Dh| \neq 0$. Therefore,
\begin{align*}
P_0 \not \in \big(\Gamma_B \setminus \{P_2\} \big)
\end{align*}
Hence, by \eqref{eq:PGE Boundary decomposition} and \eqref{eq:PGE P_1 & P_2},  $P_0$ is a point on the front boundary $\Gamma_F$
\begin{align*}\label{eq:PGE x_0 1-4a}
P_0=(x_0,y_0,z_0)=(1-4\alpha,y_0,z_0) \in \Gamma_F. \tag{4.12}
\end{align*}

\medskip

\noindent \textit{Step 3 : Distance between $P_0$ and $P_2$.} In this step, we will estimate $z_2-z_0$ in terms of $\alpha$.

We recall that the Grim reaper curve $\bd\Delta_{\alpha}$ is the graph of the function $f_\alpha (x)$ defined by \eqref{eq:PGE Grim reaper curve f_a} and $\Delta_{\epsilon,\alpha}$ is a subset of $\overline{\Delta}_{\epsilon,\alpha}\eqqcolon \Delta_{\alpha}+t_{\epsilon,\alpha}e_{z}$. Hence, $\bd\overline{\Delta}_{\epsilon,\alpha}$ is the graph of the function $f_{\epsilon,\alpha}$ defined by
\begin{align*}\label{eq:PGE f_ep,a}
f_{\epsilon,\alpha}(x)=t_{\epsilon,\alpha}+f_\alpha(x)=t_{\epsilon,\alpha}-\frac{\cR}{4\pi\alpha} \log \cos \bigg(\frac{4\alpha \pi }{\cR}\Big(x-1+3\alpha+\frac{\cR}{2\alpha}\Big)\bigg).\tag{4.13}
\end{align*}

By definition of $\varphi_\alpha$, there exists a unique point $(\bar x_0, \bar z_0) \in \bd\Delta_{\epsilon,\alpha}$ such that  
\begin{align*}\label{eq:PGE d(P_0,bar P_0)}
d((x_0,z_0) ,\overline{\Delta}_{\epsilon,\alpha})=d((x_0,z_0),(\bar x_0, \bar z_0)) \leq 2\alpha . \tag{4.14} 
\end{align*}
We know $x_0=x_2=1-4\alpha$ by \eqref{eq:PGE P_1 & P_2} and \eqref{eq:PGE x_0 1-4a}. Hence, for all $x \in [\bar x_0,x_2]$, we can derive from \eqref{eq:PGE f_ep,a} the following inequality 
\begin{align*}\label{eq:PGE f'(x_2)}
f'_{\epsilon,\alpha}(\bar x_0) \leq f'_{\epsilon,\alpha}(x) \leq f'_{\epsilon,\alpha}(x_2) =\tan \bigg(\frac{4\pi\alpha}{\cR}\Big(x_0-1+3\alpha+\frac{\cR}{2\alpha}\Big)\bigg)=\cot \frac{4\pi\alpha^2}{\cR} \,. \tag{4.15}
\end{align*}

Therefore, combining \eqref{eq:PGE d(P_0,bar P_0)} and \eqref{eq:PGE f'(x_2)} yields
\begin{align*}\label{eq:PGE z_2-z_0}
z_2-z_0 & \leq (\bar z_0-z_0) +(z_2- \bar z_0)\leq  2\alpha +(z_6- \bar z_0) \leq 2\alpha+ \int^{x_2}_{\bar x_0} f'_{\epsilon,\alpha}(x) dx  \leq  2\alpha + 2 \alpha \cot \frac{4\pi\alpha^2}{\cR} \, . \tag{4.16}
\end{align*}

\begin{figure}[h]
\begin{tikzpicture}\label{fig:PGE Level sets}
[x=1cm,y=1cm] \clip(-6,-1) rectangle (6,3);
\draw[color=blue,line width=0.04cm] (-2,0.5) -- (-2,2.5) ;
\draw[color=blue,dashed] (-2,-1) -- (-2,3) ;
\draw (0,0) parabola (6,9);
\draw (0,0) parabola (-6,9);
\draw[color=red,thick] plot [domain=-4:-2] (\x,{(0.375*(\x)^2 -0.5)});
\draw[dashed] plot [domain=-4:0] (\x,{-1.5*\x-2});
%\draw[fill=black] (-1,-0.5) circle  (0.06 cm);
\draw (0,0.5) parabola (6,12.5);
\draw (0,0.5) parabola (-6,12.5);
\draw[fill=black] (-1,0.833) circle  (0.06 cm);
\draw[fill=black] (-2,1.833) circle  (0.06 cm);
\draw[fill=black] (-1,0.25) circle  (0.06 cm);
\draw[fill=black] (-2,2.5) circle  (0.06 cm);
\draw[fill=black] (-2,0.5) circle  (0.06 cm);
\draw[dashed] (-1,-1) -- (-1,3) ;
\draw[color=red,fill=red] (-2,1) circle  (0.06 cm);
\draw[->] (-4,0.5) -- (-4,1.5);
\draw[->] (-3.5,1) -- (-4.5,1);
\begin{scriptsize}
\draw[color=black] (-0.7,0.9) node[scale=1.3] {$z_4$};
\draw[color=black] (-0.7,0.3) node[scale=1.3] {$z_5$};
\draw[color=black] (-0.3,2) node {$x=1-5\alpha$};
\draw[color=blue] (-2.7,-0.4) node {$x=1-4\alpha$};
\draw[color=black] (-2.3,0.4) node[scale=1.3] {$z_1$};
\draw[color=black] (-2.3,0.9) node[scale=1.3] {$z_0$};
\draw[color=black] (-1.7,1.9) node[scale=1.3] {$z_6$};
\draw[color=black] (-1.7,2.6) node[scale=1.3] {$z_2$};
\draw[color=black] (-3.8,1.4) node {$z$};
\draw[color=black] (-4.4,0.8) node {$x$};
\draw[color=black] (0.7,1.3) node {$u(\cdot \, ,-1+\epsilon)$};
\draw[color=black] (-3.5,2.1) node {$u(\cdot,y_0)$};
\draw[color=red] (-2.5,2.4) node {$f^{\,0}_{\epsilon,\alpha}$};
\end{scriptsize}
\end{tikzpicture}
\caption{Level sets of the solution $h$}
\end{figure}
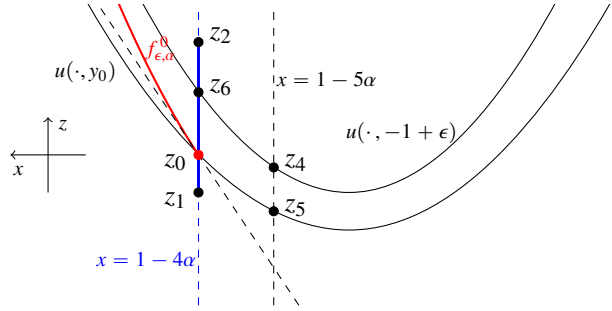

\noindent \textit{Step 4 : Partial derivative $\bd_x u$ bound at the contact point $P_0$.} We can consider the level set $L_{y_0}(\varphi_{\epsilon,\alpha})$ as the graph of a convex function $ f^{\, 0}_{\epsilon,\alpha}:[1-4\alpha,1-\alpha) \to \mathbb{R}$, namely $L_{y_0}(\varphi_{\epsilon,\alpha})=\{(x,f^{\, 0}_{\epsilon,\alpha}(x)):x \in [1-4\alpha,1-\alpha)\}$. Then, by the definitions  of $\varphi_{\alpha}$ and $(\bar x_0 ,\bar z_0)$, we have $(f^{\, 0}_{\epsilon,\alpha})'(x_0)=f'_{\epsilon,\alpha}(\bar x_0)$. Thus, \eqref{eq:PGE f'(x_2)} yields the bound
\begin{align*}
(f^{\, 0}_{\epsilon,\alpha})'(x_0) \leq \cot (4\pi\alpha^2 / \cR) \,.
\end{align*}

On the other hand,  \eqref{eq:PGE Contact point} implies $L_{y_0}(\varphi_{\epsilon,\alpha}) \prec L_{y_0}(h)$, namely $u(x,y_0) \leq f^{\, 0}_{\epsilon,\alpha}(x)  $ holds for all $x \in [1-4\alpha,1-\alpha)$. Therefore,
\begin{align*}\label{eq:PGE u_x bound at P_0}
\bd_x u(x_0,y_0)\leq (f^{\, 0}_{\epsilon,\alpha})'(x_0) \leq \cot (4\pi\alpha^2 / \cR) . \tag{4.17}
\end{align*}

\medskip

\noindent \textit{Step 5 : Partial derivative $\bd_x u$ bound at the given point.} 
We define the  points $P_3,P_4,P_5$ on $\Sigma$ by 
\begin{align*}
P_3  & =(x_3,y_3,z_3)\eqqcolon (1-5\alpha,y_0,u(1-5\alpha,y_0)) , \\
P_4  & =(x_4,y_4,z_4)\eqqcolon (1-5\alpha,-1+\epsilon,u(1-5\alpha,-1+\epsilon))  ,\\
P_5  & =(x_5,y_5,z_5)\eqqcolon (1-4\alpha,-1+\epsilon,u(1-4\alpha,-1+\epsilon)).
\end{align*} 
Since we know $P_0,P_3 \in L^y_{y_0}(\Sigma)$, the inequality \eqref{eq:PGE u_x bound at P_0} and the convexity of $u$ give
\begin{align*}
z_0-z_3 = \int^{x_0}_{x_3} \bd_x u(x,y_0)dx \leq \int^{x_0}_{x_3} \bd_x u(x_0,y_0)dx =\alpha \big( \bd_x u(x_0,y_0) \big) \leq \alpha \cot \frac{4\pi\alpha^2}{\cR} \,.
\end{align*}

By adding \eqref{eq:PGE z_2-z_0} and the inequality above, we obtain
\begin{align*}\label{eq:PGE z_2-z_3}
z_2-z_3 \leq  2 \alpha +3 \alpha \cot (4\pi\alpha^2 /\cR). \tag{4.18}
\end{align*}

On the other hand, \eqref{eq:PGE Contact point} implies that $L_{-1+\epsilon}(\varphi_{\epsilon,\alpha}) \prec L_{-1+\epsilon}(h)$. Therefore, $x_2=x_5=1-4\alpha$, $(x_2,z_2) \in L_{-1+\epsilon}(\varphi_{\epsilon,\alpha})$, and $(x_5,z_5) \in L_{-1+\epsilon}(h)$ guarantee
\begin{align*}
z_2=\varphi_{\epsilon,\alpha}(x_2) \geq u(x_2,-1+\epsilon)=u(x_5,-1+\epsilon)=z_5 .
\end{align*}
Also, the convexity and symmetry of $\Sigma$ give that 
\begin{align*}
z_3 \leq z_4
\end{align*}
Thus, subtracting the inequalities above yields $z_5-z_4 \leq z_2-z_3$. Applying \eqref{eq:PGE z_2-z_3} , we have
\begin{align*}
z_5-z_4 \leq 2\alpha+3\alpha \cot (4\pi\alpha^2 /\cR).
\end{align*}

Hence,  the desired result follows by the following computation
\begin{align*}
z_5-z_4 = \int^{x_5}_{x_4} \bd_x u(x,-1+\epsilon) dx \geq  \int^{x_5}_{x_4} \bd_x u(x_4,-1+\epsilon) dx = \alpha  \big(\bd_x u(1-5\alpha,-1+\epsilon) \Big).
\end{align*}

\end{proof}

\section{Distance from the tip to flat sides}

Let $\Sigma$ be the  translating  solution to  the Gauss curvature flow over the square $\Omega$ as   in 
Theorem \ref{thm:INT Flat sides}.  In this final section we will show that this solution has flat sides, as stated in 
Theorem \ref{thm:DTF Distance} below. To this end, we will study  the distance from the tip of a solution $\Sigma$ to each point 
on the {\em free boundary},  that is the boundary of the flat sides. To estimate this distance one needs  to establish  a {\em gradient bound} for solutions to  the equation 
\eqref{eq:INT GCFeq} at {\em a certain point} near the flat sides. Since the gradient bound depends on the global structure
of  $\Omega$, we will establish an integral estimate by deriving 
a  separation of variables structure from \eqref{eq:INT GCFeq} as in  the proof of the following Lemma.

\begin{lemma}[Gradient bound]\label{lemma:DTF Gradient estimate}
Let a domain $\Omega$ and a solution $u$ satisfy the conditions in \emph{Theorem \ref{thm:PGE Gradient bound for level set} }. Assume that  $[a,b] \times [-1,-1+\sigma] \subset (-1,1) \times [-1,-\frac{1}{2})$, 
for some  constants $a,b,\sigma$. Then, there exists a point $x_0 \in [a,b]$ satisfying
\begin{align*}
-\bd_y u (x_0,-1+\sigma) \leq  \sqrt{\frac{M\cR}{2\pi\sigma(b-a)}} \,,
\end{align*} 
where $\displaystyle  M=\sup_{y \in (0,\sigma)}\sup_{x\in (a,b)}|\bd_x u| (x,-1+y)$.
\end{lemma}

\begin{proof}
Since we have $\cA(\Omega) <\cR$, the following   inequality holds
\begin{align*}
\frac{u_{yy}u_{xx}}{(1+u_y^2)^{\frac{3}{2}}}  \geq \frac{u_{yy}u_{xx}-u_{xy}^2}{(1+u_x^2+u_y^2)^{\frac{3}{2}}} = \frac{ \det D^2u}{(1+|Du|^2)^{\frac{3}{2}}} = \frac{2\pi}{\cA}> \frac{2\pi}{\cR} \, .
\end{align*}
Combining the above inequality with Holder inequality yields
\begin{align*}
\Big(\int^b_a \frac{u_{yy}}{(1+u_y^2)^{\frac{3}{2}}} dx \Big) \Big( \int^b_a u_{xx}\, dx \Big) > \Big( \int^b_a (2\pi/\cR)^{\frac{1}{2}} \; dx\Big)^2=\frac{2\pi}{\cR}(a-b)^2.
\end{align*}
On the other hand, for $y \in (0,\sigma)$ the following holds
\begin{align*}
\int^b_a u_{xx}(\cdot,-1+y)\, dx = u_x(b,-1+y)-u_x(a,-1+y)  \leq 2M.
\end{align*}
Hence,
\begin{align*}
\frac{1}{b-a}\int^b_a \int_{-1}^{-1+\sigma} \frac{u_{yy}}{(1+u_y^2)^{\frac{3}{2}}} dydx \geq \frac{\pi \sigma}{M\cR}(b-a).
\end{align*}
Therefore,  there exists a constant $x_0 \in [a,b]$ satisfying
\begin{align*}
\int_{-1}^{-1+\sigma} \frac{u_{yy}}{(1+u_y^2)^{\frac{3}{2}}}(x_0,\cdot\,) dy \geq \frac{\pi \sigma}{M\cR}(b-a).
\end{align*}
Moreover, by $|u_{y}| \leq (1+u_y^2)^{\frac{1}{2}}$, we have 
\begin{align*}
\int_{-1}^{-1+\sigma} \frac{u_{yy}}{(1+u_y^2)^{\frac{3}{2}}}(x_0,\cdot\,) \, dy = \frac{u_{y}}{(1+u_y^2)^{\frac{1}{2}}}(x_0,\cdot\,) \Bigg |^{-1+\sigma}_{-1}\leq  \frac{u_{y}}{(1+u_y^2)^{\frac{1}{2}}}(x_0,-1+\sigma)+1.
\end{align*}
Hence, at $(x_0,-1+\sigma)$, the following holds
\begin{align*}
\frac{|u_{y}|}{(1+u_y^2)^{\frac{1}{2}}}(x_0,-1+\sigma)=\frac{-u_{y}}{(1+u_y^2)^{\frac{1}{2}}}(x_0,-1+\sigma)\leq 1-\frac{\pi \sigma}{M\cR}(b-a).
\end{align*}
We may assume $u_y \neq 0$ and take the reciprocal of the above inequality.
\begin{align*}
1+u_y^{-2}(x_0,-1+\sigma) \geq \big(1-\frac{\pi \sigma}{M\cR}(b-a)\big)^{-2} \geq \big(1+\frac{\pi \sigma}{M\cR}(b-a)\big)^{2} \geq 1+\frac{2\pi \sigma}{M\cR}(b-a) 
\end{align*}
which implies to the desired result. 
\end{proof}

The distance between the tip of the translating solution $\Sigma$ over the square and its flat sides is estimated in the
following result.

\begin{theorem}[Distance between the tip and flat sides]\label{thm:DTF Distance}
Let $\Omega$ and $u$ satisfy the conditions in  \emph{Theorem  \ref{thm:PGE Gradient bound for level set}}. Given $\alpha \in (0,1/6)$, the following holds 
\begin{equation*}
\inf_{5\alpha\leq x\leq 6\alpha}u(1-x,-1)-\inf_{\Omega} u \leq  (4+\cR\alpha^{-2})\text{diam}(\Omega) .
\end{equation*}
\end{theorem}
 
 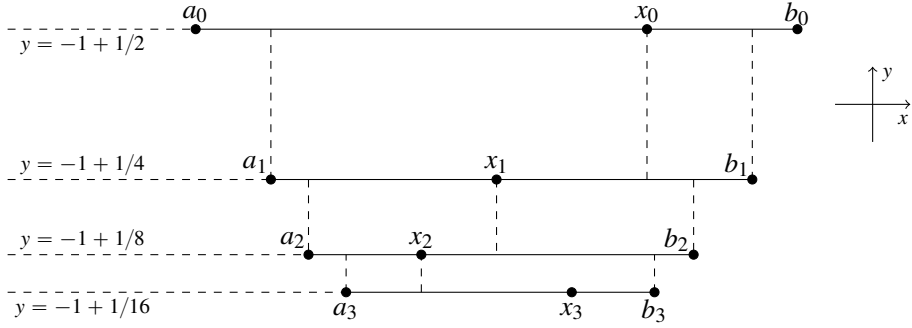
\begin{figure}[h]
\begin{tikzpicture}\label{fig:DTF Convering Points}
[x=1cm,y=1cm] \clip(-5.5,0) rectangle (6.5,4.5);
\draw (-3,4) -- (5,4);
\draw[dashed] (-5.5,4) -- (-3,4);
\draw (-2,2) -- (4.4,2);
\draw[dashed] (-5.5,2) -- (-2,2);
\draw (-1.5,1) -- (3.62,1);
\draw[dashed] (-5.5,1) -- (-1.5,1);
\draw[dashed] (-2,4) -- (-2,2);
\draw[dashed] (4.4,4) -- (4.4,2);
\draw[dashed] (3,4) -- (3,2);
\draw[dashed] (-1.5,2) -- (-1.5,1);
\draw[dashed] (3.62,2) -- (3.62,1);
\draw[dashed] (1,2) -- (1,1);
\draw (-1,0.5) -- (3.1,0.5);
\draw[dashed] (-5.5,0.5) -- (-1,0.5);
\draw[dashed] (-1,0.5) -- (-1,1);
\draw[dashed] (3.1,0.5) -- (3.1,1);
\draw[dashed] (0,0.5) -- (0,1);
\draw[->] (5.5,3) -- (6.5,3);
\draw[->] (6,2.5) -- (6,3.5);
\draw[fill=black] (-3,4) circle (0.06cm);
\draw[fill=black] (5,4) circle (0.06cm);
\draw[fill=black] (-2,2) circle (0.06cm);
\draw[fill=black] (4.4,2) circle (0.06cm);
\draw[fill=black] (-1.5,1) circle (0.06cm);
\draw[fill=black] (3.62,1) circle (0.06cm);
\draw[fill=black] (-1,0.5) circle (0.06cm);
\draw[fill=black] (3.1,0.5) circle (0.06cm);
\draw[fill=black] (3,4) circle (0.06cm);
\draw[fill=black] (1,2) circle (0.06cm);
\draw[fill=black] (0,1) circle (0.06cm);
\draw[fill=black] (2,0.5) circle (0.06cm);
\begin{scriptsize}
\draw[color=black] (-4.5,3.8) node {$y=-1+1/2$};
\draw[color=black] (-4.5,2.2) node {$y=-1+1/4$};
\draw[color=black] (-4.5,1.2) node {$y=-1+1/8$};
\draw[color=black] (-4.5,0.3) node {$y=-1+1/16$};
\draw[color=black] (6.4,2.8) node {$x$};
\draw[color=black] (6.2,3.4) node {$y$};
\draw[color=black] (-3,4.2) node[scale=1.3] {$a_0$};
\draw[color=black] (5,4.2) node[scale=1.3] {$b_0$};
\draw[color=black] (3,4.2) node[scale=1.3] {$x_0$};
\draw[color=black] (1,2.2) node[scale=1.3] {$x_1$};
\draw[color=black] (-2.2,2.2) node[scale=1.3] {$a_1$};
\draw[color=black] (4.2,2.2) node[scale=1.3] {$b_1$};
\draw[color=black] (0,1.2) node[scale=1.3] {$x_2$};
\draw[color=black] (-1.7,1.2) node[scale=1.3] {$a_2$};
\draw[color=black] (3.4,1.2) node[scale=1.3] {$b_2$};
\draw[color=black] (2,0.25) node[scale=1.3] {$x_3$};
\draw[color=black] (-1,0.25) node[scale=1.3] {$a_3$};
\draw[color=black] (3.1,0.25) node[scale=1.3] {$b_3$};
\end{scriptsize}
\end{tikzpicture}
\caption{Converging points on domain $\Omega$}
\end{figure}
 
\begin{proof}
We begin by setting $a_0=1-6\alpha$, $b_0=1-5\alpha$, $\sigma_0=\frac{1}{2}$, and 
\[M=\sup_{y \in (0,\sigma_0)}\sup_{x \in [a_0,b_0]}|\bd_x u|(x,-1+y).\]
Then, by Lemma \ref{lemma:DTF Gradient estimate}, there exists a point $x_0 \in [a_0,b_0]$ satisfying
\begin{align*}\label{eq:DTF gradient bound at -1/2}
-\bd_y u(x_0,-1+1/2) \leq  2^{-\frac{1}{2}}(M\cR/ \pi \alpha)^{\frac{1}{2}}. \tag{5.1}
\end{align*} 
We choose an interval $[a_1,b_1]$ satisfying $x_0 \in [a_1,b_1] \subset [a_0,b_0]$ and $b_1-a_1=2^{-1/3}\alpha$. Then, for $\sigma_1 = 2^{-2}$, we have $\displaystyle  \sup_{y \in (0,\sigma_1)}\sup_{x\in [a_1,b_1]}|\bd_x u|(x,-1+y) \leq M$. Hence, Lemma \ref{lemma:DTF Gradient estimate} gives a point $x_1 \in [a_1,b_1]$ satisfying
\begin{align*}
-\bd_y u(x_1,-1+1/2^2) \leq  2^{-\frac{1}{2}+\frac{2}{3}}( M\cR/ \pi \alpha)^{\frac{1}{2}}.
\end{align*}
By setting $\sigma_n=2^{-1-n}$, we can inductively choose intervals $[a_n,b_n]$ satisfying $x_{n-1}\in [a_n,b_n]\subset [a_{n-1},b_{n-1}]$ and $b_n-a_n= 2^{-n/3}\alpha$ so that we obtain a point $x_n \in [a_n,b_n]$ satisfying 
\begin{align*}
-\bd_y u(x_n,-1+1/2^{n+1}) \leq  2^{-\frac{1}{2}+\frac{2}{3}n}( M\cR/ \pi \alpha)^{\frac{1}{2}}.
\end{align*}
Then, integrating along $y$ yields
\begin{align*}
\bigg| u(x_n,-1+\frac{1}{2^{n+1}})-u(x_n,-1+\frac{1}{2^n})\bigg| \leq \int^{-1+1/2^n}_{-1+1/2^{n+1}} -\bd_y u(x_n,y)dy \leq  2^{-\frac{3}{2}-\frac{1}{3}n}(M\cR/ \pi \alpha)^{\frac{1}{2}}.
\end{align*}
On the other hand, $x_{n-1},x_n \in [a_n,b_n]$ and $b_n-a_n=2^{-n/3}\alpha$ imply
\begin{align*}
\bigg| u(x_n,-1+\frac{1}{2^n})-u(x_{n-1},-1+\frac{1}{2^n}) \bigg|=\bigg| \int^{x_n}_{x_{n-1}}\bd_x u(x,-1+\frac{1}{2^n}) dx\bigg|
 \leq 2^{-\frac{1}{3}n}M\alpha.
\end{align*}
Therefore, 
\begin{align*}
\bigg| u(x_n,-1+\frac{1}{2^{n+1}})-u(x_{n-1},-1+\frac{1}{2^n})\bigg| \leq  2^{-\frac n3}\big(( M\cR/ 8\pi \alpha)^{\frac{1}{2}}+M\alpha\big).
\end{align*}
By definition of $x_n$, the sequence $\{x_n\}_{n \in \mathbb{N}}$ converges to a point $\bar x \in [a_0,b_0]$. Hence, we can sum up the inequality above for all $n\in \mathbb{N}$ so that we have
\begin{align*}\label{eq:DTF distant up to flat side}
u(\bar x , -1) -u(x_0,-\frac{1}{2})\leq \sum_{n=1}^{\infty}2^{-\frac{n}{3}}\big(( M\cR/ 8\pi \alpha)^{\frac{1}{2}}+M\alpha\big)\leq 4 \big(( M\cR/ 8\pi \alpha)^{\frac{1}{2}}+M\alpha\big). \tag{5.2}
\end{align*}
Next, we consider the  linear function
$$f(x,y)=\big(\bd_x u(x_0,-1/2)\big)\, (x-x_0) +\big(\bd_y u(x_0,-1/2)\big)\, (y+1/2)  + u(x_0,-1/2)$$ whose graph is the tangent hyperplane of $\Sigma$ at $(x_0,-1/2, u(x_0,-1/2))$. Then, the convexity of $\Sigma$ gives $f(x,y)\leq u(x,y)$ in $\Omega$. Hence, \eqref{eq:DTF gradient bound at -1/2} and definition of $M$ prove the following for $(x,y)\in \Omega$
\begin{align*}
u(x_0,-\frac{1}{2}) -u(x,y) \leq f(x_0,-\frac{1}{2}) -f(x,y)\leq \text{diam}(\Omega) \big(M+( M\cR/ 2\pi \alpha)^{\frac{1}{2}}\big).
\end{align*}
Thus, \eqref{eq:DTF distant up to flat side}, $\alpha \leq 1/6$, $\cR > \cA(\Omega) \geq 2$, and the inequality give
\begin{align*}
u(\bar x , -1) -\inf_{\Omega}u \leq 2\text{diam}(\Omega)\big(M+( M\cR/ 2\pi \alpha)^{\frac{1}{2}}\big) .
\end{align*}
Moreover,   Theorem \ref{thm:PGE Gradient bound for level set} shows $$u_x \leq 2+3\cot(4\pi\alpha^2/\cR)$$ in $[a_0,b_0]\times (-1,-1+\sigma_0)$.
 In addition,  $b_0=1-6\alpha \geq 0 $ and the convexity of $u$ imply $u_x(x,y)\geq u_x(0,y)$ for $x \geq b_0$. Thus, Theorem \ref{thm:PGE Gradient bound for level set} shows again $$-u_x \leq 2+3\cot (4\pi/25\cR)\leq 2+3\cot(4\pi \alpha^2/\cR)$$ in $[a_0,b_0]\times (-1,-1+\sigma_0)$. Namely,
\begin{align*}
M \leq 2+3\cot(4\pi\alpha^2/\cR)\leq 2+\frac{3\cR}{4\pi \alpha^2}.
\end{align*}
Hence,
\begin{align*}
\frac{u(\bar x , -1) -\inf_{\Omega}u}{\text{diam}(\Omega)}  \leq 4+\frac{3\cR}{2\pi \alpha^2}+2\Big(\frac{\cR}{\pi\alpha}+\frac{3\cR^2}{8\pi^2\alpha^3}\Big)^{\frac{1}{2}}\leq 4+\frac{\cR}{\alpha^2}.
\end{align*}
Thus, $\bar x  \in [a_0,b_0]$ implies the desired result.
\end{proof}

We will now conclude  the proof of  the main Theorem \ref{thm:INT Flat sides}. This readily  follows  from the following result. 

\begin{theorem}[Existence of flat sides]\label{thm:DTF main thm}
Let $\Omega$ be a convex open bounded domain in $\mathbb{R}^2$, and let $u$ be a solution to \eqref{eq:INT GCFeq} on $\Omega$. Then, the corresponding solution $\Sigma$ to \eqref{eq:INT GCF} is  class $C^{1,1}_{\text{loc}}$.

 Suppose that the boundary $\bd \Omega$ contains a line segment $L=\{tp+(1-t)q:p,q \in \bd \Omega , t\in (0,1)\}$. Then, there exists a function $\bar u:L \to \mathbb{R}$ such that
\begin{equation*}
\bar u(\vec{x}_0)= \lim_{\vec{x} \to \vec{x}_0}u(\vec{x}).
\end{equation*}
\end{theorem}

\begin{proof}
Let $\{\Omega_n\}_{n \in \mathbb{N}}$ be a sequence of bounded convex open sets $\mathbb{R}^2$ such that 
\begin{itemize}
\item $\{\Omega_n\}$ monotonically decreases to $\Omega$, namely $\Omega_{n+1}\subset \Omega_n$ and $\Omega_n \to \Omega$,
\item each boundary $\bd \Omega_n$ is a strictly convex and smooth hypersurface in $\mathbb{R}^n$.
\end{itemize}
Then, we let $\{u_n\}_{n \in \mathbb{N}}$ and $\{\Sigma_n \}_{n \in \mathbb{N}}$ be the sequence of corresponding solutions of \eqref{eq:INT GCFeq} with $\inf u_n=0$ and their graphs, respectively. We denote the convex hull of $\Sigma_n$ by 
$E_n=\{tX+(1-t)Y: X,Y \in \Sigma_n, t \in [0,1]\}$, and define a convex body $E$ by
\begin{align*}
E= \bigcap_{n \in \mathbb{N}} E_n .
\end{align*} 
Let us show that the boundary $\bd E$ is a graph over $\Omega$. Given a point $\vec{x}_0\in \Omega$ and a direction $e \in \mathbb{R}^2$, we denote by $\Omega^{\vec{x}_0,e}$ the subset set $\{\vec{x}\in\Omega:\langle \vec{x}-\vec{x}_0,e \rangle \geq 0\}$, and denote $\cA^{\vec{x}_0}_{\inf}(\Omega)=\inf \{\cA(\Omega^{\vec{x}_0,e}): e \in \mathbb{R}^2,|e|=1\}$. Then, $\cA^{\vec{x}_0}_{\inf}(\Omega)$ is a uniform lower bound for the area of $\Omega_n^{\vec{x}_0}$ in Lemma \ref{lemmma: OP Height bound by level set}. Moreover, $\cA(\Omega_1)$ and $\text{diam}(\Omega_1)$ are uniform upper bounds for $\cA(\Omega_n)$ and $\text{diam}(\Omega_n)$, respectively. Hence, Lemma \ref{lemmma: OP Height bound by level set} and $\inf u_n=0$ give a uniform upper bound $U(\vec{x}_0)$ for $u_n(\vec{x}_0)$, namely $(\vec{x}_0,U(\vec{x}_0)) \in E$. Therefore, $\bd E$ is a graph on $\Omega$.  Since Theorem \ref{thm:OP Optimal regularity} gives the uniform curvature estimates for $\Sigma_n$, the limit $\Sigma=\bd E$ is a $C_{\text{loc}}^{1,1}$ solution to \eqref{eq:INT GCF} defined on $\Omega$.

\medskip

Next, we assume that $\bd\Omega$ contains a line segment $L=\{(x,y_0):a<x<b\}$. Given a point $x_0 \in [a_0,b_0] \subset (a,b)$, we define a sequence of numbers $\{y_n\}$ by $\langle Du_n(x_0,y_n),e_2\rangle=0 $. Then, Lemma \ref{lemmma: OP Height bound by level set} and Theorem \ref{thm:OP Graph in region} yield a uniform lower bound $r$ for $|y_n-\bar y_n^{\pm} |$ depending on $a_0,b_0$, where $(x_0,\bar y_n^{\pm}) \in \bd\Omega_n$. Hence, for sufficiently large $n$, we have $ \langle Du_n ,e_2\rangle \neq 0$ in $[a_0,b_0]\times [y_0,y_0+\frac{r}{2}]$. Namely, the surfaces $\Sigma_n$ are convex graphs with respect to $e_2$ in $[a_0,b_0]\times [y_0,y_0+\frac{r}{2}]\times \mathbb{R}$. 

Now, we may assume $y_0=-1, \frac{r}{2}\geq 1, [-1,1] \subset [a_0,b_0]$ by scaling and translating the solutions. Then, $\Sigma_n$ satisfies the conditions in Theorem \ref{thm:PGE Gradient bound for level set}. By Theorem \ref{thm:DTF Distance}, given $\alpha \in (0,1/6)$ there exists a sequence of point $\{x_n^+\}$ such that $x_n^+\in [1-6\alpha,1-5\alpha]$ and 
\begin{align*}
u_n(x_n^+,-1)\in (4+\cA(\Omega_1)\alpha^{-2})\text{diam}(\Omega_1)=C_\alpha.
\end{align*}
Hence, the limit $\bar x^+ \in [1-6\alpha,1-5\alpha]$ of a subsequence of $x_n^+$ satisfies $(x_n^+,-1,C_\alpha) \in E$. In the same manner, there exists a number $\bar x^- \in [-1+5\alpha,-1+6\alpha]$ such that $(x_n^-,-1,C_\alpha) \in E$. Then, the convexity of $E$ yields  $(x,-1,C_\alpha) \in E$ for all $x\in [-1+6\alpha,1-6\alpha] \subset [\bar x^-,\bar x^+]$, which implies the desired result.
\end{proof}

\centerline{\bf Acknowledgements}

\smallskip

\noindent K. Choi has been partially supported by NSF grant DMS-1811267.\\
\noindent P. Daskalopoulos has been partially supported by NSF grant DMS-1600658.\\
\noindent Ki-Ahm Lee Lee has been supported by the National Research Foundation
of Korea (NRF) grant funded by the Korean government (MSIP) (No. NRF2017R1A2A2A05001376)
Ki-Ahm Lee also holds  a joint appointment with the Research Institute of Mathematics of Seoul National University.

\end{document}